\documentclass[11pt]{amsart}
\usepackage[all]{xy}
\usepackage{enumitem}
\usepackage{mathtools}
\usepackage{mathrsfs}
\usepackage{amssymb}
\usepackage{amsmath,amsfonts,amsthm}
\usepackage{graphicx}
\usepackage{upgreek}
\usepackage[toc,page]{appendix}
\usepackage{bbm}
\usepackage{esint}
\usepackage{scalerel}
\usepackage{stackengine,wasysym}
\usepackage{hyperref}
\usepackage[dvipsnames]{xcolor}
\hypersetup{colorlinks=true, linktocpage=true,citecolor=Emerald, linkcolor=RedViolet, urlcolor=RedViolet}

\usepackage{geometry}
\newcommand\restr[2]{{
		\left.\kern-\nulldelimiterspace 
		#1 
		\right|_{#2} 
}}

\numberwithin{equation}{section}
\newtheorem{teo}{Theorem}[section]
\newtheorem{lema}[teo]{Lemma}
\newtheorem{prop}[teo]{Proposition}
\newtheorem{coro}[teo]{Corollary}
\newtheorem{defi}[teo]{Definition}

\newtheorem{obs}[teo]{Remark}

\def\Xint#1{\mathchoice
	{\XXint\displaystyle\textstyle{#1}}%
	{\XXint\textstyle\scriptstyle{#1}}%
	{\XXint\scriptstyle\scriptscriptstyle{#1}}%
	{\XXint\scriptscriptstyle\scriptscriptstyle{#1}}%
	\!\int}
\def\XXint#1#2#3{{\setbox0=\hbox{$#1{#2#3}{\int}$}
		\vcenter{\hbox{$#2#3$}}\kern-.5\wd0}}

\def\dashint{\Xint-}

\title{Isosystolic Inequalities for Holomorphic Chains in $\mathbb{C}P^{\MakeLowercase{n}}$}

\author{Luciano L. Junior}
\address{Università di Trento, Dipartimento di Matematica, via Sommarive 14, 38123 Trento, Italy}
\email{luciano.luzzijunior@unitn.it} 

\begin{document}
	
	\begin{abstract}
		We introduce the \emph{holomorphic \(k\)-systole} of a Hermitian metric on \(\mathbb{C}P^n\), defined as the infimum of areas of homologically non-trivial holomorphic \(k\)-chains. Our main result establishes that, within the set of Gauduchon metrics, the Fubini-Study metric locally minimizes the volume-normalized holomorphic \((n-1)\)-systole. As an application, we construct Gauduchon metrics on \(\mathbb{C}P^2\) arbitrarily close to the Fubini-Study metric whose homological \(2\)-systole is realized by non-holomorphic chains.         
	\end{abstract}
	
	\maketitle
	
	\setcounter{tocdepth}{1}
	\tableofcontents
	
	\section{Introduction}
	
	In complex differential geometry, an interesting question is to identify sufficient conditions under which a minimal stable submanifold of a compact complex manifold must be holomorphic. More generally, to establish when locally mass-minimizing currents are holomorphic chains—that is, currents defined as finite sums of irreducible holomorphic subvarieties with integer multiplicities.  
	
	 A classical instance of this problem is present in the work of B. Lawson and J. Simon \cite{Lawson_Simon_73}, where they prove that in the complex projective space $(\mathbb{C}P^n,J,g_{FS})$ endowed with the canonical complex structure and the Fubini-Study metric, every closed stable $2k$-current is necessarily a holomorphic chain. Later, J. King and R. Harvey-B. Shiffman completely characterize homologically volume-minimizing locally rectifiable currents as (positive) holomorphic chains in closed Kähler manifolds, provided the homology class defined by the current admits a (positive) holomorphic chain as representative \cite[Theorem 2.3]{Harvey_Shiffman74}. See also~\cite{Alexander97} for a further generalization of the aforementioned characterization. It is also noteworthy that non-holomorphic stable minimal surfaces representing $(1,1)$-classes can be constructed in Kähler-Einstein manifolds \cite{Arezzo00}.    
	 
	 Therefore is natural to inquire which previous results generalize when weakening the K\"ahler condition, while imposing stricter geometric constraints. To approach this problem, we introduce the \emph{holomorphic \(k\)-systole}---defined as the infimum of the area among homologically non-trivial holomorphic \(k\)-chains---and analyze its behavior relative to the classical \emph{(homology) \(2k\)-systole}---i.e., the infimum of the area among all homologically non-trivial \(2k\)-chains. Our analysis focuses explicitly on Hermitian metrics of \((\mathbb{C}P^n, J)\) with its fixed canonical complex structure, and on the examination of codimension $2$ cycles. 
	 
	Namely, letting $\mathscr{H}(\mathbb{C}P^n)$ represent the set of Hermitian metrics on $(\mathbb{C}P^n,J)$, we denote the holomorphic $(n-1)$-systole of $g \in \mathscr{H}(\mathbb{C}P^n)$ by:  
	\begin{equation}\label{holosystole}  
		\mathrm{Sys}^{\mathrm{Hol}}_{n-1}(g) =\inf \left\{ \mathrm{Area}_g(C) : C \text{ is a non-trivial holomorphic } (n-1)\text{-chain} \right\}.  
	\end{equation}  
	We also introduce the \emph{normalized holomorphic $(n-1)$-systole}:     
	\begin{equation}\label{rhofunctional}
		\rho(g)\doteq \frac{\mathrm{Sys}^{\mathrm{Hol}}_{n-1}(g)}{\mathrm{Vol}(g)^{\frac{n-1}{n}}}.
	\end{equation}    
	
	The analysis of the holomorphic systole will be through variational methods applied to the scale-invariant functional $\rho$ on $\mathscr{H}(\mathbb{C}P^n)$. To overcome the non-smoothness of systolic functionals, we construct a family of smooth (upper) support functions via the \emph{family of equators}, i. e., the family of complex codimension $1$ linear complex projective subspaces of $\mathbb{C}P^n$. This smoothness enables first- and second-order variational analysis, as done in~\cite{junior2023balanced}.
	
	Following~\cite{lucas_coda_andre}, we identify the first variation of our support functions with the \emph{Radon Transform} associated with the family of equators (see Proposition~\ref{FVFofMsigma}). Hence, via the application of special integral geometric formulas and a first-order Taylor expansion we establish that $1$-parameter perturbations of the Fubini-Study metric transverse to the kernel of this operator have strictly decreasing normalized holomorphic $(n-1)$-systole (see Corollary~\ref{firstorderdesc}). A key observation is that we can characterize this kernel as the tangent directions of \emph{Gauduchon metrics}—namely, Hermitian metrics $g\in \mathscr{H}(\mathbb{C}P^n)$ whose fundamental form $\omega_g$ satisfies $dd^c\omega_g^{n-1}=0$, with $d^c \doteq J^{-1}dJ$ denoting the complex-twisted exterior derivative. We denote the space of such metrics by $\mathscr{G}(\mathbb{C}P^n)$.       
	
	This observation imposes second-order analysis of $\rho$ restricted to $\mathscr{G}(\mathbb{C}P^n)$ in order to characterize the holomorphic systole's behavior of all Hermitian metrics around the Fubini-Study metric. The Gauduchon equation enables explicit description of the holomorphic systole through the $dd^c$-Lemma, bypassing support functions and leading to the following conclusions. The space $\mathscr{K}(\mathbb{C}P^n)$ of Kähler metrics is contained in the critical point set of $\rho$ restricted to the space of Gauduchon metrics, with semi-positive definite Hessians at each Kähler metric. Taylor's expansion theorem then yields our main result:          
	{
	\renewcommand{\theteo}{A}
	\begin{teo}\label{maintheoremA}
		Fix $n\geq 2$. There exists an open set $ \mathscr{K}(\mathbb{C}P^n) \subset \mathcal{U} \subset \mathscr{G}(\mathbb{C}P^n)$, in the ${C}^{3}$-topology, such that for every metric $g\in \mathcal{U}$, 
		$$\rho(g) \geq \rho(g_{FS}).$$
		Moreover, $g \in \mathcal{U}$ satisfies the equality if and only if $g \in \mathscr{K}(\mathbb{C}P^n)$.
	\end{teo}       
	}
	As a consequence of our local description of the normalized holomorphic $(n-1)$-systole, we are able to demonstrate that the result of J. King and R. Harvey--B. Shiffman no longer holds when the K\"ahler condition is relaxed to Gauduchon in $(\mathbb{C}P^2,J)$. Moreover, employing M. Gromov and M. Berger's characterization of the Fubini-Study metric as a local maximum for the normalized $2$-systole in $\mathbb{C}P^2$ (see~\cite[\S 0.2.B]{gromov_pshol}), we establish the following local converse to the aforementioned theorem.     
	{
	\renewcommand{\theteo}{B}
	\begin{coro}\label{maintheoremB}
		There exists a $C^{\infty}$-neighborhood $\mathcal{U} \subset \mathscr{G}(\mathbb{C}P^2)$ of the Fubini-Study metric with the following property: given $g \in \mathcal{U}$, if any $2$-chain realizing the $2$-systole of $g$ is holomorphic, then $g$ is Kähler. In particular, for every sufficiently small $C^{\infty}$-neighborhood of the Fubini-Study metric in $\mathscr{G}(\mathbb{C}P^n)$, there exists a Gauduchon metric whose $2$-systole is realized by a non-holomorphic chain.      
	\end{coro}
	} 
	Another significant manifestation of our central question occurs in the context of the duality between stable minimal immersions and holomorphic maps, as presented in a classical theorem of S.-T. Yau and Y.-T. Siu \cite{Yau_Siu1980}, which states: if $(M^{2n},J,g)$ is a K\"ahler manifold with positive holomorphic bisectional curvature, then every conformal branched immersion $f: \mathbb{C}P^1 \to M$ that is minimal and stable is either holomorphic or anti-holomorphic. Therefore, our ability to select Gauduchon metrics with positive sectional curvature, implicitly established in our previous results, leads to the fact that this result cannot be extended to the Gauduchon setting.  
	{
	\renewcommand{\theteo}{C}	 
	\begin{coro}\label{maintheoremC}
		For every sufficiently small $C^{\infty}$-neighborhood of the Fubini-Study metric in $\mathscr{G}(\mathbb{C}P^2)$, there exist a Gauduchon metric $g$ with positive sectional curvature and a minimal stable conformal branched immersion $f_g: \mathbb{C}P^1 \to (\mathbb{C}P^2,g)$ that is neither holomorphic or anti-holomorphic.     
	\end{coro}
	}

	Up to this point, Gauduchon metrics have functioned as a natural analytical framework for understanding systolic behavior on $\mathbb{C}P^n$ and studying obstructions to extending Kähler results. However, the general setting established by P. Gauduchon's seminal work~\cite{Gauduchon77}---which guarantees that every closed Hermitian manifold of complex dimension greater than one admits a unique (up to homothety) Gauduchon metric in its conformal class--- endows them with intrinsic value as a general tool in complex geometry. Therefore, it is natural to employ the work developed here to characterize geometric properties of $\mathscr{G}(\mathbb{C}P^n)$. This analysis is especially valuable in complex dimension 2, where dimensionality forces equivalences between classical non-Kähler conditions: Strong Kähler with Torsion~\cite{Bismut89}, Strong Gauduchon~\cite{Popovici13}, and Gauduchon metrics coincide, while the balanced condition~\cite{Michelsohn_1982} implies Kähler. For a ample discussion on non-Kähler geometry see~\cite{Verbitsky_Liviu24}.   
	
	Our first application draws parallels with the classification theory of \emph{generalized Zoll families} developed in~\cite{lucas_coda_andre, junior2023balanced}. For Gauduchon metrics, the $dd^c$-Lemma immediately yields constant area along the family of equators of the complex projective space. We prove via the theory developed to study the Radon transform and consequently the first-variation of our support function that this constancy characterizes the set $\mathscr{G}(\mathbb{C}P^n)$.     
	{
	\renewcommand{\theteo}{D}	
	\begin{prop}\label{maintheoremD}
		A Hermitian metric $g \in \mathscr{H}(\mathbb{C}P^n)$ is Gauduchon if and only if the function defined by the area along the family of equators is constant.
	\end{prop} 
	}
	The previously mentioned Theorem of P. Gauduchon can be restated as the existence of a projection $\Pi_{\mathscr{G}}: \mathscr{H}(\mathbb{C}P^n) \to \mathscr{G}(\mathbb{C}P^n)$, assigning to each Hermitian metric the unique volume-normalized Gauduchon metric in its conformal class. As we point out in Section~\ref{section:Applications}, this projection is continuous in the $C^{\infty}$-topology. Meanwhile, in~\cite[\S4.A.2]{gromov_systole} M. Gromov provided a general panorama to manufacture isosystolic inequalities within conformal classes around the Fubini-Study metric. Therefore, combing his ideas with Theorem~\ref{maintheoremA} we prove the following.            
	{
	\renewcommand{\theteo}{E}	
	\begin{teo}\label{maintheoremE}
		There exists a $C^{\infty}$-neighborhood of the Fubini-Study metric $\mathcal{U} \subset \mathscr{G}(\mathbb{C}P^n)$, for which the following inequalities hold: 
		\begin{enumerate}[label=\alph*),ref=(\alph*)]
			\item\label{mainthrm2itema} $\rho(g) \leq \rho(\Pi_{\mathscr{G}}(g))$, for every $ g \in \Pi_{\mathscr{G}}^{-1}(\mathcal{U})$.
			\item\label{mainthrm2itemb} $\rho(\Pi_{\mathscr{G}}(g)) \geq \rho(g_{FS})$, for every $g \in \Pi_{\mathscr{G}}^{-1}(\mathcal{U})$.
		\end{enumerate}    
	\end{teo}      
	}
	To conclude our introduction, we propose to the reader a possible research program for generalizing Corollary~\ref{maintheoremB} to higher dimensions. Note first that our proof deeply relies on understanding the behavior of the classical 2-systole for Gauduchon metrics, which in complex dimension~2 is covered by the Gromov-Berger theorem (see~\cite[\S 0.2.B]{gromov_pshol}, and also~\cite[Theorem A]{junior2023balanced})---a result admitting no direct generalization to higher dimensions. However, for \( g \in \mathscr{G}(\mathbb{C}P^n) \) (\( n \geq 3 \)), the vector field defined via duality by \( \mathrm{div}_g\,\omega_g \) provides simultaneous descent directions for all equators of \( \mathbb{C}P^n \), as it represents their mean vector field up to the action of \( J \). Thus, an application of the minimal submanifold construction via the inverse function theorem introduced by B.~White~\cite{BWhite91} (see also~\cite{lucas_coda_andre}), combined with this preferred choice of section for the normal bundle of the equators, may lead to a fruitful study of the systole of Gauduchon metrics in a neighborhood of the Fubini--Study metric and, consequently, to a generalization of our result.           
	
	\subsection*{Organization of the paper} The first part of Section~\ref{section:ClassificationofGauduchonMetrics} is devoted to setting up the notation and recalling all the tools of complex differential geometry that we will need. Later, we study a generalization of the classical Radon transform to higher-order tensors and, as a consequence, prove Proposition~\ref{maintheoremD}. In Section~\ref{section:VariationalFormulae}, we define our smooth support functions and derive the formulae for their first- and second-order variations. Once these formulas are established, Section~\ref{section:IsosystolicInequalities} proceeds to prove Theorem~\ref{maintheoremA} via Taylor expansion. Finally, in Section~\ref{section:Applications}, we use the results developed in Sections~\ref{section:VariationalFormulae} and~\ref{section:IsosystolicInequalities} to prove Corollaries~\ref{maintheoremB} and~\ref{maintheoremC} as well as Theorem~\ref{maintheoremE}.        
	\subsection*{Acknowledgements} This project has received funding from the European Research Council (ERC) under the European Union’s Horizon 2020 research and innovation programme (grant agreement No. 947923). The author is also grateful to Lucas Ambrozio and Alessandro Carlotto for enriching suggestions on earlier drafts of this paper.

	\section{Classification of Gauduchon Metrics}\label{section:ClassificationofGauduchonMetrics}
	
	As previously mentioned in the introduction, the objective of this section is two-folded, introduce and study the Radon transform and trough the developed theory characterize the set of Gauduchon metrics as stated in Proposition~\ref{maintheoremD}. The foundation for both objectives lays in the geometrical properties of the family of linear equators of the projective space, which is already know to hold interesting properties. In fact, in dimension $2$, it serves as a toy model for the compact moduli space of pseudo-holomorphic curves (see \cite{gromov_pshol}). Also, in higher dimensions, it serves as a Zoll family for balanced metrics (see \cite{junior2023balanced} and \cite{lucas_coda_andre}). In what follows, we elaborate on the structure of the family of linear equators. But first, we set up notation.               
	
	\subsection*{Preliminaries}  The inequalities in Theorem~\ref{maintheoremA} and \ref{maintheoremE} will be derived through applications of first- and second-order variational principles. To frame this analysis, we need to equip the space of Hermitian metrics on $\mathbb{C}P^n$ and its derivatives with a Banach manifold structure.
	
	First, recall that the canonical complex structure $J$ induces the identification:
	\begin{align}\label{formxmetrics}
		\begin{split}
			\mathcal{J}: \Omega^{1,1}_+(\mathbb{C}P^n) &\to \mathscr{H}(\mathbb{C}P^n) \\
			\omega &\mapsto g_\omega(\cdot,\cdot) \doteq \omega(\cdot,J\cdot),
		\end{split}     
	\end{align}
	where $\Omega^{p,p}_+(\mathbb{C}P^n)$ denotes the space of smooth sections of the open cone of \emph{positive} $(p,p)$-forms:
	\[
	\Lambda_+^{p,p} \doteq \left\{\alpha \in \Lambda_{\mathbb{R}}^{p,p} \,\bigg|\, 
	\begin{array}{ll}
		\alpha(v_1,\ldots,v_p,Jv_1,\ldots,Jv_p) > 0 \text{ for every} \\
		\text{linearly independent set } \{v_j,Jv_j\}_{j=1}^p
	\end{array}
	\right\}.
	\]
	Here, $\Lambda_{\mathbb{R}}^{p,p}$ denotes the bundle of real $(p,p)$-forms.
	
	By endowing $\Omega^{1,1}_{\mathbb{R}}(\mathbb{C}P^n)$ with the Hölder topology, we may view $\mathscr{H}(\mathbb{C}P^n)$ as an open subset of this space, thereby inheriting a natural structure of a Banach manifold.
	
	More precisely, let $\nu \in (0,1)$. For any vector bundle $E \to \mathbb{C}P^n$, we will denote by $C^{k,\nu}(E)$ the space of sections of $E$ with Hölder regularity $C^{k,\nu}$. It is well-known that $\left(C^{k,\nu}(E), \|\cdot\|_{C^{k,\nu}} \right)$ forms a Banach vector space. Consequently, defining $\mathscr{H}^{k,\nu}(\mathbb{C}P^n)$ as the space of Hermitian metrics with regularity $C^{k,\nu}$ with the topology provided of the inclusion on the space of symmetric tensors, the map in~\eqref{formxmetrics} extend to a homeomorphism:
	\[
	\mathcal{J}: C^{k,\nu}(\Lambda^{1,1}_+) \to \mathscr{H}^{k,\nu}(\mathbb{C}P^n).
	\]
	This equips the space of Hermitian metrics with a Banach manifold structure, inherited from the open set $C^{k,\nu}(\Lambda^{1,1}_+)$, of the Banach vector space $C^{k,\nu}(\Lambda^{1,1}_\mathbb{R})$. From this moment forward $\nu \in (0,1)$ will be fixed.         
	
	We also define the sets of Kähler, balanced and Gauduchon metrics with Hölder regularity: 
	\begin{align*}
		\mathscr{K}^{2,\nu}(\mathbb{C}P^n) &\doteq \{g \in \mathscr{H}^{2,\nu}(\mathbb{C}P^n) : d \omega_g = 0\};\\
		\mathscr{B}^{2,\nu}(\mathbb{C}P^n) &\doteq \{g \in \mathscr{H}^{2,\nu}(\mathbb{C}P^n) : d \omega^{n-1}_g = 0\};\\
		\mathscr{G}^{2,\nu}(\mathbb{C}P^n) &\doteq \{g \in \mathscr{H}^{2,\nu}(\mathbb{C}P^n) : dd^c \omega^{n-1}_g = 0\}. 		
	\end{align*}       
	
	\begin{obs}
		Note that $\mathscr{K}(\mathbb{C}P^n)$ is $C^{2,\nu}$-dense in $\mathscr{K}^{2,\nu}(\mathbb{C}P^n)$. This follows from the Hodge decomposition and the $dd^c$-Lemma for low regularity metrics (cf.~\cite{Morrey1956}, \cite[\S1.5]{Joyce_SpHolo}). The same reasoning applies to balanced and Gauduchon metrics.     
	\end{obs}

	Now we prove that the set of Kähler, balanced and Gauduchon metrics also can be endowed with a structure of Banach manifold. 
	
	\begin{prop}\label{chartsofKBG}
		The spaces $\mathscr{K}^{2,\nu}(\mathbb{C}P^n)$, $\mathscr{B}^{2,\nu}(\mathbb{C}P^n)$, and $\mathscr{G}^{2,\nu}(\mathbb{C}P^n)$ of $C^{2,\nu}$ K\"ahler, balanced, and Gauduchon metrics on $\mathbb{C}P^n$ have a smooth Banach manifold structure, for which the canonical inclusions are topological embeddings.
	\end{prop} 
	\begin{proof}
		Following~\cite{Michelsohn_1982} (see also \cite[Proposition 5.2]{junior2023balanced}) the smooth map
		\begin{align*}
			C^{2,\nu}\big(\Lambda_+^{1,1}\big) &\to C^{2,\nu}\big(\Lambda_+^{n-1,n-1}\big) \\
			\omega &\mapsto \omega^{n-1},
		\end{align*}
		is a diffeomorphism between Banach manifolds. Moreover, it induces homeomorphisms
		\[
		\mathscr{B}^{2,\nu}(\mathbb{C}P^n) \ni g \mapsto \omega_g \in C^{2,\nu}_d(\Lambda_+^{1,1})
		\quad \text{and} \quad
		\mathscr{G}^{2,\nu}(\mathbb{C}P^n) \ni g \mapsto \omega_g \in C^{2,\nu}_{dd^c}(\Lambda_+^{1,1}),
		\]
		where $C^{2,\nu}_d(\Lambda_+^{1,1})$ and $C^{2,\nu}_{dd^c}(\Lambda_+^{1,1})$ denote the spaces of $d$-closed and $dd^c$-closed forms, respectively. These mappings provide global charts, inducing Banach manifold structures on the corresponding spaces of metrics. The case of K\"ahler metrics follows analogously via Corollary $5.4$ of \cite{junior2023balanced} and the subsequent observation.  
	\end{proof}

	We finish our recapitulation recalling some important differential operators in complex geometry. 
	
	Let $g$ be a Hermitian metric on $\mathbb{C}P^n$. As previously mentioned, $d^c$ denotes the exterior derivative twisted by the complex structure. We also define $\delta_g, \delta^c_g : \Omega^{\bullet}(\mathbb{C}P^n) \to \Omega^{\bullet-1}(\mathbb{C}P^n)$ as the \emph{codifferential operators} induced by $g$ and its twisted version, respectively. The \emph{Lefschetz operator} $L_g : \Omega^{\bullet}(\mathbb{C}P^n) \to \Omega^{\bullet+2}(\mathbb{C}P^n)$ is defined as wedge product with the K\"ahler form associated to $g$, moreover its adjoint is denoted by $\Lambda_g : \Omega^{\bullet}(\mathbb{C}P^n) \to \Omega^{\bullet-2}(\mathbb{C}P^n)$, and its called the \emph{dual Lefschetz operator}. Hodge theory provides the decomposition $C^{2,\nu}(\Lambda^\bullet) = \mathbb{R} \omega_g  \oplus \mathrm{Im}(d) \oplus \mathrm{Im}(\delta_g)$. The canonical projections associated with this decomposition are denoted by $\Pi_g$, $\Pi_{d,g}$, and $\Pi_{\delta,g}$, respectively. When these operators are written without subscripts, it is understood that they are associated with the Fubini-Study metric. Also, from this moment forward the Fubini-Study metric will be normalized to satisfy $\mathrm{Vol}(\mathbb{C}P^n,g_{FS})=(n!)^{-1}$, or equivalently $\int_{\mathbb{C}P^n} \omega_{FS}^n = 1$. 
	
	\subsection*{The Generalized Radon Transform}
		
	Prior to introducing the Radon transform, we recall that the family of equators is parameterized again by the complex projective space in the following sense: every complex codimension $1$ linear complex projective subspace can be described by $\Sigma_{\sigma} \doteq \{p \in \mathbb{C}P^n : p \perp \sigma \}$, for some $\sigma \in \mathbb{C}P^n$. A natural description of this parametrization is via the \emph{incidence set}:  
	$$ \mathcal{I} \doteq \{(p, \sigma) \in \mathbb{C}P^n \times \mathbb{C}P^n : p \in \Sigma_\sigma\} \subset \mathbb{C}P^n \times \mathbb{C}P^n.$$
	
	In fact, the incidence set endowed with the canonical projections on the first and second coordinates defines the \emph{double fibration} (see \cite[Definition $2.6$]{paiva_fernandes}):     
	\begin{equation}\label{doublefibration}
		\begin{split}	
			\begin{xy}\xymatrix{
					& \mathcal{I} 
					\ar[dl]_{\kappa}  \ar[dr]^{\upsilon} & \\
					\mathbb{C}P^{n}&   & \mathbb{C}P^{n}. 
				}
			\end{xy} 
		\end{split}
	\end{equation}   
	
	With this language the family of equators can be expressed as $\Sigma_{\sigma} = \kappa (\upsilon^{-1}(\sigma))$, for each $\sigma\in \mathbb{C}P^n$. Moreover, we define the \emph{dual family} as $\Xi_{p} \doteq \upsilon(\kappa^{-1}(p))$, for each $p \in \mathbb{C}P^n$. Clearly, the dual family consist of linear equators in the space of parameters. 
	
	The classical \emph{Radon transform} associated to the family of equators $\{\Sigma_{\sigma}\}_{\sigma\in \mathbb{C}P^n}$ in $(\mathbb{C}P^n, g_{FS})$ is defined as the operator $\mathcal{R}: C^{0}(\mathbb{C}P^n) \to C^{0}(\mathbb{C}P^n)$ acting on a continuous function $f$ as follows:
	\[
	\mathcal{R}(f)(\sigma) = \dashint_{\Sigma_\sigma} f d\mathrm{A}_{g_{FS}},
	\]
	 An extensive literature regarding the Radon transform for two-point symmetric spaces, due to S. Helgason, can be found in \cite{HelgasonBook1999}. 
	
	It will also be useful to define the \emph{dual Radon transform}, which is the unique operator satisfying:
	\begin{equation}\label{dualradondef}
		\int_{\mathbb{C}P^n} \mathcal{R}(\phi) \psi \, d\mathrm{Vol}_{g_{\mathrm{FS}}} = \int_{\mathbb{C}P^n} \phi \mathcal{R}^*(\psi) \, d\mathrm{Vol}_{g_{\mathrm{FS}}},
	\end{equation}
	for every $\phi, \psi \in C^{0}(\mathbb{C}P^n)$. Endowing the incidence set with the product metric and applying the co-area formula together with the fact that $(\mathbb{C}P^n,g_{FS})$ is a two-point symmetric space, we see that the dual Radon transform is explicitly given by the formula $\mathcal{R}^*(f)(p) = \dashint_{\Xi_p} f d\mathrm{A}_{g_{FS}}$. 
	
	In \cite{HelgasonBook1999},  S. Helgason proved the \emph{inversion formula} for $\mathcal{R}$ in the set of smooth functions, which allow us to infer it is injective.
	
	\begin{prop}\label{Radonker}
		The Radon transform $\mathcal{R}: C^{0}(\mathbb{C}P^n) \to C^{0}(\mathbb{C}P^n)$ has trivial kernel. 
	\end{prop}  
	\begin{proof}
		Fix $f \in \mathrm{Ker}(\mathcal{R})$, and let $\phi \in C^{\infty}(\mathbb{C}P^n)$ be a smooth function. By the inversion formula (see \cite[Chapter $3$, Theorem $2.2$]{HelgasonBook1999}) there exists a polynomial $p$ such that $\phi = \mathcal{R}^*(p(\Updelta_{g_{FS}})\mathcal{R}(\phi))$, hence setting $\psi= p(\Updelta_{g_{FS}})\mathcal{R}(\phi) \in C^{\infty}(\mathbb{C}P^n)$, the following identities hold:   	
		$$\int_{\mathbb{C}P^n}f \phi \, d\mathrm{Vol}_{g_{\mathrm{FS}}} = \int_{\mathbb{C}P^n}f \mathcal{R}^*(\psi) \, d\mathrm{Vol}_{g_{\mathrm{FS}}} = 
		\int_{\mathbb{C}P^n}\mathcal{R}(f) \psi \, d\mathrm{Vol}_{g_{\mathrm{FS}}} =0.$$
		Thus $f = 0$ by density of smooth functions.      
	\end{proof}
	
	The Radon transform of a function~$f$ at~$\sigma \in \mathbb{C}P^n$ can be thought of as the integral over the equator~$\Sigma_\sigma$ of the $(2n-2)$-form~$f \omega^{n-1}_{\text{FS}}$, which is defined on all of~$\mathbb{C}P^n$. This observation suggests a natural generalization of the Radon transform to arbitrary $(2n-2)$-form.
	
	In fact, we define the \emph{generalized Radon transform} associated to the family of equators $\{\Sigma_{\sigma}\}_{\sigma\in \mathbb{C}P^n}$ in $\mathbb{C}P^n$ as the operator $\mathcal{R}: C^{2,\nu}(\Lambda^{2n-2}) \to C^{2,\nu}(\mathbb{C}P^n)$ acting on the $(2n-2)$-forms $\xi$ by: 
	\begin{equation}\label{genradontrans}
		\mathcal{R}(\xi)(\sigma) = \int_{\Sigma_{\sigma}}{ \xi} .
	\end{equation}

 	An important tool associated to the classical Radon transform is the \emph{integral geometric formula} associated to the family of equators $\{\Sigma_{\sigma}\}_{\sigma\in \mathbb{C}P^n}$, which can be recover from the explicit formula for the dual Radon transform. Indeed, putting $\psi=1$ in the equation \eqref{dualradondef} we obtain: 
	\begin{equation}\label{IGF}
		\int_{\mathbb{C}P^n}\left( \dashint_{\Sigma_\sigma} f \,d\mathrm{A}_{g_{FS}}\right) \,d \mathrm{Vol}_{g_{FS}} = \int_{\mathbb{C}P^n} f \,d \mathrm{Vol}_{g_{FS}}.
	\end{equation} 
	
	In the next proposition we generalize this integral geometric formula for tensors.
	\begin{prop}[Integral Geometric Formula]\label{IGFfortensors}
		For every $\xi \in C^{2,\nu}(\Lambda^{2n-2})$, there holds:
		\[
		\dashint_{\mathbb{C}P^n} \left( \int_{\Sigma_\sigma} \xi \right) d\mathrm{Vol}_{g_{\text{FS}}}
		= \int_{\mathbb{C}P^n} \xi \wedge \omega_{g_{\text{FS}}}.
		\]
	\end{prop}
	\begin{proof}
		The projections $\kappa,\upsilon: \mathcal{I} \to \mathbb{C}P^n$ induces maps $\kappa_*,\upsilon_*:\Omega^\bullet(\mathcal{I}) \to \Omega^{\bullet-(2n-2)}(\mathbb{C}P^n)$ by integration along the fiber (see \cite[\S6]{bott_tu}). With this notation we can rewrite the generalized Radon transform as $\mathcal{R}(\xi)=\upsilon_*(\kappa^*\xi)$. Therefore, by the classical \emph{projection formula} \cite[Proposition $6.15$]{bott_tu}, for every $\xi \in \Omega^{2n-2}(\mathbb{C}P^n)$  we have that:  
		\begin{equation}\label{prejectionformula}
			\dashint_{\mathbb{C}P^n} \left( \int_{\Sigma_\sigma} \xi \right) d\mathrm{Vol}_{g_{\text{FS}}} = \int_{\mathbb{C}P^n}\upsilon_*(\kappa^*\xi) \wedge \omega_{FS}^{n} = \int_{\mathbb{C}P^n} \xi \wedge \kappa_*(\mu^* (\omega^n_{FS})).
		\end{equation}
		
		On the other hand, $\kappa_*(\upsilon^* (\omega^n_{FS}))$ is an $SU(n+1)$-invariant $2$-form. Hence, it is homothetic to the Fubini-Study form. However, by putting $\xi=\omega^{n-1}_{FS}$ in~\eqref{prejectionformula} we conclude that the homothety constant is one. The general result follows by density.        
	\end{proof} 
	
	Now that we have set all the machinery related to the Radon transform, we turn our attention to proving Proposition~\ref{maintheoremD}. That is, to classify the Hermitian metrics for which the area along the family of linear equators is constant.
	
	More specifically, given a Hermitian metric \( g \in \mathscr{H}(\mathbb{C}P^n) \), we define the function:
	\begin{equation}\label{areafunction}
		\begin{aligned}
			\mathcal{A}_g: \mathbb{C}P^n &\to \mathbb{R} \\
			\sigma &\mapsto \mathcal{A}_g(\sigma) \doteq \mathrm{Area}_g(\Sigma_\sigma).  
		\end{aligned}
	\end{equation}      
	
	Therefore, the aforementioned proposition is equivalent to characterizing the set of Gauduchon metrics as the set of Hermitian metrics for which the function \(\mathcal{A}_g\) is constant.
	
	The main tool that we will need is the following description of the function~\eqref{areafunction} in terms of the generalized Radon transform:
	\begin{equation}\label{A_gxR}
		\mathcal{A}_g(\sigma) =\frac{1}{(n-1)!}\, \mathcal{R}(\omega_g^{n-1})(\sigma).
	\end{equation}      
	
	In fact, the description~\eqref{A_gxR} together with the geometric integral formula yields the following result.
	   
	\begin{prop}\label{propareaxradontr}
		Let $g \in \mathscr{H}^{2,\nu}(\mathbb{C}P^n)$ be a Hermitian metric. Then, the area function $\mathcal{A}_g$ is constant if and only if the $(2n-2)$-form: $$\omega^{n-1}_{g} - \left(\int_{\mathbb{C}P^n} \omega^{n-1}_{g}\wedge \omega_{FS}\right) \omega^{n-1}_{FS}$$ is in the kernel of the generalized Radon transform.
	\end{prop}   
	\begin{proof}
		Fix $g \in \mathscr{H}^{2,\nu}(\mathbb{C}P^n)$. By the integral geometric formula, the average of $\mathcal{A}_g$ is given by: 
		$$\dashint_{\mathbb{C}P^n}\mathcal{A}_g(\sigma)\, d\mathrm{Vol}_{g_{FS}} = \frac{1}{(n-1)!} \int_{\mathbb{C}P^n} \omega_{g}^{n-1}\wedge \omega_{FS}.$$
		
		Hence, by the identity $\mathcal{R}(\omega^{n-1}_g)= {(n-1)!}\, \mathcal{A}_g $ we obtain: 
		$$\mathcal{R}\left(\omega^{n-1}_{g} - \left(\int_{\mathbb{C}P^n} \omega^{n-1}_{g}\wedge \omega_{FS}\right) \omega^{n-1}_{FS}\right) = (n-1)!\left\{\mathcal{A}_g - \dashint_{\mathbb{C}P^n}\mathcal{A}_g(\sigma)\, d\mathrm{Vol}_{g_{FS}} \right\},$$
		which concludes the proof. 
	\end{proof}
	  
	The previous reasoning motivates the following definition. The \emph{zero average Radon transform} is the operator $\mathring{\mathcal{R}}: C^{2,\nu}(\Lambda^{2n-2}) \to C^{2,\nu}(\mathbb{C}P^n)$ given by $\mathring{\mathcal{R}}(\xi) \doteq \mathcal{R}\big(\mathring{\xi} \big)$, where $\mathring{\xi}$ is the zero average part of the tensor $\xi$, which is defined by $\xi - \left(\int_{\mathbb{C}P^n} \xi \wedge \omega_{FS}\right) \omega^{n-1}_{FS}$. Next we compute the kernel of the zero average Radon transform and as a corollary we prove Theorem \ref{maintheoremD}.
	\begin{prop}\label{kernelofmeanRT}
		The kernel of the zero average Radon transform is $\ker \mathring{\mathcal{R}} = C^{2,\nu}_{dd^c}(\Lambda^{2n-2})$.
	\end{prop}
	\begin{proof}
		First, the $dd^c$-Lemma and the Stokes' Theorem implies that $C^{2,\nu}_{dd^c}(\Lambda^{2n-2}) \subset \ker \mathring{\mathcal{R}}$. Conversely, for any $\xi \in \ker \mathring{\mathcal{R}}$ fixed, we can apply the Hodge decomposition and the Kähler identities to decompose our tensor as $\xi=f \omega^{n-1}_{FS} + d \alpha + d^c \beta$, for $f\in C^{2,\nu}(\mathbb{C}P^n)$. Moreover,  if $\mathring{f}=f - \dashint f d \mathrm{Vol}_{g_{FS}}$,  we then have 
		$$0 = \mathring{\mathcal{R}}(\xi)(\sigma) = \dashint_{\Sigma_\sigma} \mathring{f} d \mathrm{A}_{g_{FS}},$$
		for every $\sigma \in \mathbb{C}P^n$. Hence, applying Proposition \ref{Radonker} we obtain that $\xi= \left(\dashint f d \mathrm{Vol}_{g_{FS}}\right) \omega^{n-1}_{g_{FS}} + d\alpha +d^c \beta$, which concludes the proof.       
	\end{proof}  
	  Finally, we move on to prove of the characterization of the Gauduchon metrics.
	  \begin{teo}\label{charofgaudmetric}
	  	A Hermitian metric $g \in \mathscr{H}^{2,\nu}(\mathbb{C}P^n)$ is Gauduchon if and only if the function $\mathcal{A}_g$ is constant.
	  \end{teo} 
  	\begin{proof}
  		Fix $g \in \mathscr{H}^{2,\nu}(\mathbb{C}P^n)$. By Propositions \ref{propareaxradontr} and \ref{kernelofmeanRT} the area function $\mathcal{A}_g$ is constant if and only if the zero average part of ${\omega_g^{n-1}}$ is $dd^c$-closed, in particular if and only if $\omega_g^{n-1}$ is $dd^c$-closed.       
  	\end{proof}
  
  	\section{Variational Formulae}\label{section:VariationalFormulae}
	
	Having characterized the Gauduchon metrics and settled the theory of the Radon transform, we move on to study the variational behavior of the holomorphic systole. The functionals $\mathrm{Sys}^{\mathrm{Hol}}_{n-1}$ and $\rho$ defined in~\eqref{holosystole} and~\eqref{rhofunctional} extend naturally to Hermitian $C^{2,\nu}$ metrics, allowing us to apply the machinery of calculus on Banach manifolds whenever the smoothness of these functionals is guaranteed. However, as is typical for systolic functionals, they may lack smoothness. To bypass this issue, we force the desired properties by working with a suitable family of \emph{(upper) smooth support functions}. In what follows, we formalize this concept for our setting and construct the desired family.   
	
	\begin{defi}\label{suppfunction}
		A function $\mathcal{S} \colon \mathscr{H}^{2,\nu}(\mathbb{C}P^n) \to \mathbb{R}$ is called a \emph{smooth (upper) support function} for $\rho$ along $\mathscr{A} \subset \mathscr{H}^{2,\nu}(\mathbb{C}P^n)$ if:
		\begin{enumerate}[label=\alph*),ref=(\alph*)]
			\item[(i)] $\mathcal{S}$ is Fréchet-differentiable on $\mathscr{H}^{2,\nu}(\mathbb{C}P^n)$;
			\item[(ii)] $\rho(g) \leq \mathcal{S}(g)$ for all $g \in \mathscr{H}^{2,\nu}(\mathbb{C}P^n)$;
			\item[(iii)] $\mathcal{S}(g) = \rho(g)$ for all $g \in \mathscr{A}$.
		\end{enumerate}
	\end{defi}    
	
	One trivial example is given by the functional area. For every given complex hypersurface \(\Sigma \subset \mathbb{C}P^n\) homologous to a linear equator, the \emph{normalized area functional} \({\mathcal{M}}_\Sigma: \mathscr{H}^{2,\nu}(\mathbb{C}P^n) \to \mathbb{R}\), defined by 
	\begin{equation*}
		\mathcal{M}_{\Sigma}(g) \doteq \frac{\mathrm{Area}_g(\Sigma)}{\mathrm{Vol}(g)^{\frac{n-1}{n}}},
	\end{equation*}
	acts as a support function for the functional \(\rho\) along the Kähler metrics.
	
	The uniform distribution of linear equators in \(\mathbb{C}P^n\) naturally induces a family of support functions. As before, let \(\Sigma_\sigma = \{p \in \mathbb{C}P^n : p \perp \sigma\}\) denote the linear equator orthogonal to \(\sigma \in \mathbb{C}P^n\). Hence, to each \(\sigma\), we can associate the support function \(\mathcal{M}_\sigma \doteq \mathcal{M}_{\Sigma_\sigma}\). Clearly, 
	\begin{equation}\label{M_sigmaxA_g}
		\mathcal{M}_\sigma(g)= \frac{\mathcal{A}_g(\sigma)}{\mathrm{Vol}(g)^{\frac{n-1}{n}}},
	\end{equation} 
	for every metric $g\in \mathscr{H}(\mathbb{C}P^n)$ and $\sigma \in \mathbb{C}P^n$.     
	
	In order to take advantage of the equidistribution of the family of linear equators, we define the \emph{mean functional} $\mathcal{M}: \mathscr{H}^{2,\nu}(\mathbb{C}P^n) \to \mathbb{R}$, which is given by the average of the family $\{\mathcal{M}_\sigma\}_{\sigma \in \mathbb{C}P^n}$, that is: 
	\begin{equation}\label{meanfunctional}
		\mathcal{M}(g) \doteq \dashint_{\mathbb{C}P^n} \frac{\mathrm{Area}_g(\Sigma_\sigma)}{\mathrm{Vol}(g)^{\frac{n-1}{n}}} \, d\mathrm{Vol}_{g_{\mathrm{FS}}}(\sigma).
	\end{equation}
	
	In what follows, we prove that the mean functional $\mathcal{M}$ serves as a support function for $\rho$ along the set of Gauduchon metrics; in particular, this shows that $\rho$ is smooth along this subspace. To this end, we first rewrite $\rho$ under the identification \eqref{formxmetrics} and then apply the integral geometric formula for the family of equators in $(\mathbb{C}P^n, g_{FS})$ (Proposition \ref{IGFfortensors}). This yields: 
	\begin{equation}\label{definitionofM}
		\mathcal{M}(g_\omega)  =
		\frac{(n!)^{\frac{n-1}{n}}}{(n-1)!}\, \frac{\int_{\mathbb{C}P^n} \omega^{n-1}\wedge {\omega_{FS}} }{\left( \int_{\mathbb{C}P^n} \omega^n\right)^{\frac{n-1}{n}}} .
	\end{equation}
	This formula immediately provides the smoothness of $\mathcal{M}$, as it is given by the product of multi-linear bounded operators. Also, since it is the average of the areas of holomorphic submanifolds, it defines a support function for $\rho$ along the Kähler metrics. It remains to show that $\mathcal{M}$ is also a support function along the set of Gauduchon metrics, which follows from the next proposition.   
	
	\begin{prop}\label{systolexM,MsigmaonGaud}
		Fix $g \in \mathscr{G}^{2, \nu}(\mathbb{C}P^n)$, then 
		$$ \mathrm{Sys}^{\mathrm{Hol}}_{n-1}(g) = \mathrm{Area}_g(\Sigma_\sigma),$$
		for every $\sigma \in \mathbb{C}P^n$. In particular, $\mathcal{M}_\sigma$, $\mathcal{M}$ and $\rho$ coincide along $\mathscr{G}^{2,\nu}(\mathbb{C}P^n)$.   
	\end{prop} 	
	\begin{proof}
		Fix \( g \in \mathscr{G}^{2,\nu}(\mathbb{C}P^n) \). By definition, \( dd^c \omega_g^{n-1} = 0 \). Combining the \( dd^c \)-Lemma (see \cite{huybrechts2005complex}, Lemma~3.A.22, \cite{Joyce_SpHolo}, \S1.5) with the fact that the de Rham cohomology \( \mathrm{H}^\bullet_{\text{dR}}(\mathbb{C}P^n, \mathbb{R}) \) is generated by the Fubini-Study form, we may decompose the fundamental form of \( g \) as $ \omega_g^{n-1} = a \, \omega_{\text{FS}}^{n-1} + d\alpha + d^c\beta$, where \( a \) is a positive real constant. Therefore, for any holomorphic \( (2n-2) \)-chain \( C \), Stokes' Theorem then gives:
		\[
		\mathrm{Area}_g(C) = \frac{1}{(n-1)!} \int_C \left( a \, \omega_{\text{FS}}^{n-1} + d\alpha + d^c\beta \right) = a \, \mathrm{Area}_{g_{\text{FS}}}(C) \geq a \, \mathrm{Area}_{g_{\text{FS}}}(\Sigma_\sigma) = \mathrm{Area}_g(\Sigma_\sigma).
		\]  
		Here, we use the well-known fact that the \( (2n-2) \)-systole of \( (\mathbb{C}P^n, g_{\text{FS}}) \) is attained by any linear equator \( \Sigma_\sigma \).      
 	\end{proof}
    
    Now that we have settled the properties of our support functions, we can proceed with computing their derivatives. We begin by observing that, using equations \eqref{A_gxR} and \eqref{M_sigmaxA_g}, the first variation of the area functional \(\mathcal{M}_\sigma\) can be described in terms of the generalized Radon transform.
     
   The previous observation was used by the authors of \cite{lucas_coda_andre}, in the context of the equators of the sphere, to prove the existence of generalized Zoll families in higher dimensions. We will apply this idea, together with the theory developed in Section \ref{section:ClassificationofGauduchonMetrics}, to show that for a generic \(1\)-parameter perturbation of the Fubini–Study metric the normalized holomorphic $(n-1)$-systole is strictly decreasing.

	Moreover, from Proposition \ref{kernelofmeanRT}, we will derive that the mean functional, when restricted to the space of Gauduchon metrics, has the Kähler metrics as its critical points. Therefore, after we have established the formulae for the first variation, we will proceed to compute the second variation of \(
	\bigl.\mathcal{M}\bigr|_{\mathscr{G}^{2,\nu}(\mathbb{C}P^n)}
	\) in order to prove Theorem \ref{maintheoremA}.
	
   \subsection*{First Variation Formulae and Consequences}
   
 We proceed to establish the first variational formulas for the support functionals
 \(\mathcal{M}_\sigma\) and \(\mathcal{M}\).
 
  \begin{prop}[First Variational Formula for $\mathcal{M}_{\sigma}$]\label{FVFofMsigma}
	Let $g \in \mathscr{H}^{2,\nu}(\mathbb{C}P^n)$, $h \in T_{g} \mathscr{H}^{2,\nu}(\mathbb{C}P^n)$ and $\sigma \in \mathbb{C}P^n$, then  
   \begin{equation}\label{1variationFormForM_sigma}
	 \restr{d \mathcal{M}_{\sigma}}{g} \cdot h = \frac{(n!)^{\frac{n-1}{n}}}{(n-2)!}\left\{\int_{\Sigma_\sigma}\eta\wedge \omega_g^{n-2} - \left(\int_{\mathbb{C}P^n}\eta \wedge \omega_g^{n-1} \right) \int_{\Sigma_\sigma} \omega_{g}^{n-1} \right\},
   \end{equation}
		where $\eta \doteq h(J\cdot, \cdot)$. In particular, $\restr{d \mathcal{M}_{\sigma}}{g_{FS}} \cdot h = \mathring{\mathcal{R}}\left(\eta \wedge \omega^{n-2}_{FS}\right)(\sigma)$.    
	\end{prop}
    \begin{proof}
    	The formula \eqref{1variationFormForM_sigma} follows from a simple computation. On the other hand, by the normalization of the Fubini-Study metric, for every $\eta = h(J\cdot,\cdot)$ with $h\in T_{g_{FS}} \mathscr{H}^{2,\nu}(\mathbb{C}P^n)$ we have
    	$$\restr{d \mathcal{M}_{\sigma}}{g_{FS}} \cdot h = \frac{(n!)^{\frac{n-1}{n}}}{(n-2)!} \left\{ \int_{\Sigma_{\sigma}}\eta \wedge \omega^{n-2}_{FS} - \int_{\mathbb{C}P^n}\eta \wedge \omega^{n-1}_{FS} \right\}.$$
    	
   	Now, if $ a \omega^{n-1}_{FS} + d\alpha + \delta \theta$ is the Hodge decomposition of $\eta\wedge \omega^{n-2}_{FS}$, we obtain by Stokes' Theorem that   
   	$$ \int_{\Sigma_{\sigma}}\eta \wedge \omega^{n-2}_{FS} - \int_{\mathbb{C}P^n}\eta \wedge \omega^{n-1}_{FS} = \int_{\Sigma_\sigma} \delta \theta - \int_{\mathbb{C}P^n} \delta \theta  \wedge \omega_{FS}.$$
   	However, by integration by parts, $\int_{\mathbb{C}P^n} \delta \theta  \wedge \omega_{FS}$ is zero. We conclude the proof by noting that the zero average part of $\eta\wedge \omega^{n-2}_{FS}$ is $d\alpha + \delta \theta$ and the factor $d\alpha$ is in the kernel of the zero average Radon transform by Proposition~\ref{kernelofmeanRT}.         
    \end{proof} 
	
	The first interesting consequence of the first variational formula is that the Fubini-Study metric is a critical point for the support function $\mathcal{M}$, in fact is the only critical point up to homothety.  
	
	\begin{coro}[First Variational Formula for $\mathcal{M}$]\label{FVFforM}
		The Fubini-Study metric is the only critical point for $\mathcal{M}: \mathscr{H}^{2,\nu}(\mathbb{C}P^n) \to \mathbb{R}$, up to homothety.  
	\end{coro} 
	\begin{proof}
		Fix $g \in \mathscr{H}^{2,\nu}(\mathbb{C}P^n)$, $h \in T_{g} \mathscr{H}^{2,\nu}(\mathbb{C}P^n)$ and $\eta = h(J\cdot,\cdot)$. By the first variation formula and the integral geometric formula (Proposition \ref{IGFfortensors}), we obtain:
		\begin{align*}
			\left. d\mathcal{M} \right|_{g} \cdot h &= 
			\frac{(n!)^{\frac{n-1}{n}}}{(n-2)!}\left\{\dashint_{\mathbb{C}P^n}\left(\left(\int_{\Sigma_\sigma}\eta\wedge \omega_g^{n-2}\right) - \left(\int_{\mathbb{C}P^n}\eta \wedge \omega_g^{n-1} \right) \left(\int_{\Sigma_\sigma} \omega_{g}^{n-1}\right)\right)\,d\mathrm{Vol}_{g_{FS}} \right\}\\
			&=\frac{(n!)^{\frac{n-1}{n}}}{(n-2)!}\int_{\mathbb{C}P^n} \left(\eta\wedge \omega_g^{n-2}\right)\wedge \left(\omega_{FS} - \left(\int_{\mathbb{C}P^n}\omega^{n-1}_{g} \wedge \omega_{FS}\right)\omega_g \right).
		\end{align*}
	
	The above formula shows that the Fubini-Study metric it is a critical point. But moreover, the map $C^{2,\nu}(\Lambda^{1,1}_{\mathbb{R}})\ni \eta \mapsto \eta\wedge \omega_g^{n-2} \in C^{2,\nu}(\Lambda^{n-1,n-1}_{\mathbb{R}})$ is surjective (see \cite{huybrechts2005complex}, Proposition $1.2.30$), hence $g$ is a critical point if and only if $\omega_{FS} = \left(\int_{\mathbb{C}P^n}\omega^{n-1}_{g} \wedge \omega_{FS}\right)\omega_g.$
	\end{proof}
	  
	  The second interesting consequence of the first variational formula is that the normalized holomorphic $(n-1)$-systole is strictly decreases for generic $1$-parameter perturbations of the Fubini-Study metric. 
	 
	 \begin{coro}\label{firstorderdesc}
	 	Let $t \mapsto g_t \in \mathscr{H}^{2,\nu}(\mathbb{C}P^n)$ be a 1-parameter deformation of $g_{FS}$ with $\left.\frac{d}{dt}g_t\right|_{t=0} \notin T_{g_{FS}} \mathscr{G}^{2,\nu}(\mathbb{C}P^n)$. Then for sufficiently small $t \neq 0$,
	 	\[
	 	\rho(g_t) < \rho(g_{FS}).
	 	\]
	 \end{coro}
	 
	 \begin{proof}
	 	Let $t \to g_{t}$ denote the $1$-parameter deformation and set $h \doteq \left.\frac{d}{dt}g_t\right|_{t=0}$. Observe that the condition $h \notin T_{g_{FS}} \mathscr{G}^{2,\nu}(\mathbb{C}P^n)$ implies, by Proposition $\ref{kernelofmeanRT}$, that the function $\mathring{\mathcal{R}}(\eta\wedge \omega^{n-2}_{FS})$ is not trivial, where  $\eta = h(J\cdot,\cdot)$. Moreover, by Corollary~\ref{FVFforM} it also has zero mean, hence there exist $\sigma_\pm \in \mathbb{C}P^n$ with $\mathring{\mathcal{R}}(\eta\wedge \omega^{n-2}_{FS})(\sigma_+) < 0$ and $\mathring{\mathcal{R}}(\eta\wedge \omega^{n-2}_{FS})(\sigma_-) > 0$. 
	 	
	 	For $t \neq 0$, we define
	 	\[\sigma_t = 
	 \begin{cases}
	 	\sigma_+ & \text{if } t > 0 \\
	 	\sigma_- & \text{if } t < 0.
	 \end{cases} \]
 	Thus, by Taylor expansion there is an $\varepsilon>0$, such that, for $0 < |t| < \varepsilon$,
	 	\[
	 	\rho(g_t) \leq \mathcal{M}_{\sigma_t}(g_t) = \rho(g_{FS}) + t\left(\mathring{\mathcal{R}}(\eta\wedge \omega^{n-2}_{FS})(\sigma_t)\right) + O(t^2)< \rho(g_{FS}),
	 	\]
	  concluding the proof.
	 \end{proof}
 	
 	\begin{obs}
 		The proof of Corollary~\ref{firstorderdesc} is inspired by~\cite[Theorem $3$]{Chavel70}. However, we note that the conclusion of this theorem is false. Indeed, Proposition $4.13$ in \cite{junior2023balanced} provides a counterexample. More concretely, the flawed step lies in the assertion that $d(\delta\theta)\wedge \omega^{n-2}_{\text{FS}}=0$ implies $\delta\theta=0$, which is false. Consequently, the strictly first-order analysis employed in that work is insufficient to fully characterize the local behavior of their functional.       
 	\end{obs}
 	
 	The previous corollary motivates a detailed analysis of the support function $\mathcal{M}$ restricted to Gauduchon metrics. We begin by specializing the first variation formula to $\mathscr{G}^{2,\nu}(\mathbb{C}P^n)$, then proceed to derive the second variation formula for our support functions. 
 	
 	\begin{prop}[First Variation Formula for $\restr{\mathcal{M}}{\mathscr{G}}$]\label{FVFforMrestrG}
 		Every Kähler metric $g \in \mathscr{K}^{2,\nu}(\mathbb{C}P^n)$ is a critical point for $\restr{\mathcal{M}}{\mathscr{G}^{2,\nu}(\mathbb{C}P^n)}: \mathscr{G}^{2,\nu}(\mathbb{C}P^n) \to \mathbb{R}$. 
 	\end{prop}    
 	\begin{proof}
 		Take $g \in \mathscr{K}^{2,\nu}(\mathbb{C}P^n)$, without loss of generality we can assume that $\int_{\mathbb{C}P^n} \omega_g^n =1$. In particular, $\omega_g^{k} \wedge \omega_{FS}$ and $\omega_{FS}^{k+1}$ are $dd^c$-cohomologous, for every $k \in \mathbb{Z}_{\geq0}$. Hence, by formula \eqref{definitionofM} the first variation of $\mathcal{M}$ in the direction $h \in T_g \mathscr{G}^{2,\nu}(\mathbb{C}P^n)$ is given by: 
 		$$\restr{d \mathcal{M}}{g} \cdot h=\frac{(n!)^{\frac{n-1}{n}}}{(n-2)!} \left\{ \int_{\mathbb{C}P^n}\eta \wedge \omega^{n-2}_{g}\wedge{\omega_{FS}} - \int_{\mathbb{C}P^n}\eta \wedge \omega^{n-1}_{g} \right\} $$  
 		
 		 Moreover, arguing as Lemma $5.6$ of \cite{junior2023balanced}, we verify that under the identification~\eqref{formxmetrics} the tangent space $T_g \mathscr{G}^{2,\nu}(\mathbb{C}P^n)$ is given by $\{\eta \in C^{2,\nu}(\Lambda^{1,1}_{\mathbb{R}}) : dd^c(\eta\wedge{\omega_g}^{n-2} )=0\}$. Then, by Stoke' Theorem and integration by parts the result follows.  
 	\end{proof}
	
	\subsection*{Second Variational Formulae}
	For subsequent applications of the Taylor expansion, it will be necessary to compute the second derivative of $\mathcal{M}$ restricted to $\mathscr{G}^{2,\nu}(\mathbb{C}P^n)$ away from its critical points. This requires carrying out the computations within the global chart for $\mathscr{G}^{2,\nu}(\mathbb{C}P^n)$ established in Proposition \ref{chartsofKBG}.
	
	Let $\mathcal{K}^{2,\nu}(\mathbb{C}P^n)$, $\mathcal{B}^{2,\nu}(\mathbb{C}P^n)$, and $\mathcal{G}^{2,\nu}(\mathbb{C}P^n)$ denote the sets $\mathscr{K}^{2,\nu}(\mathbb{C}P^n)$, $\mathscr{B}^{2,\nu}(\mathbb{C}P^n)$, and $\mathscr{G}^{2,\nu}(\mathbb{C}P^n)$ respectively, under the identification \(
	\mathscr{H}^{2,\nu}(\mathbb{C}P^n) \ni g \mapsto \omega_g \in C^{2,\nu}(\Lambda^{1,1}_+)
	\). Furthermore, $\mathcal{K}_1^{2,\nu}(\mathbb{C}P^n)$, $\mathcal{B}_1^{2,\nu}(\mathbb{C}P^n)$, and $\mathcal{G}_1^{2,\nu}(\mathbb{C}P^n)$ will denote the respectively sets where an element $\omega$ is \emph{normalized}\footnote{Since the normalized holomorphic systole is invariant under homothety there is no lost of generality in work with normalized forms.} by $\int_{\mathbb{C}P^n}\omega^n =1$.

	The global trivialization from Proposition \ref{chartsofKBG} is then given by
	\begin{align}\label{trivialization}
		\begin{split}
			\Phi: \mathcal{G}^{2,\nu}(\mathbb{C}P^n) &\to C^{2,\nu}_{dd^c}(\Lambda^{n-1,n-1}_+) \\
		\omega &\mapsto \omega^{n-1}.
		\end{split}  
	\end{align} 
	Denoting its inverse by $\Psi: C^{2,\nu}_{dd^c}(\Lambda^{n-1,n-1}_+) \to \mathcal{G}^{2,\nu}(\mathbb{C}P^n)$, we define the functional
	\begin{align*}
		\mathcal{F}:  C^{2,\nu}_{dd^c}(\Lambda^{n-1,n-1}_+) &\to \mathbb{R} \\
		\sigma &\mapsto \mathcal{F}(\sigma) \doteq \frac{(n!)^{\frac{n-1}{n}}}{(n-1)!}\frac{\int_{\mathbb{C}P^n}{\sigma\wedge \omega_{FS}}}{ \left(\int_{\mathbb{C}P^n} \sigma \wedge \Psi(\sigma) \right)^{\frac{n-1}{n}}},   
	\end{align*}  
	From equation \eqref{definitionofM} we obtain $\mathcal{M}(g_{\omega}) = \mathcal{F}(\Phi(\omega))$. Now, following step-by-step the proof of Theorem 5.12 in \cite{junior2023balanced}, we derive the subsequently second variational formula for $\mathcal{M}$ restricted to $\mathscr{G}^{2,\nu}(\mathbb{C}P^n)$. 
		
	\begin{teo}[Second variational formula of $\mathcal{F}$]\label{svfforf}
		If $\omega \in \mathcal{G}^{2,\nu}_1(\mathbb{C}P^n)$, $\eta \in T_{\omega} \mathcal{G}^{2,\nu}(\mathbb{C}P^n)$ and $\mu= \restr{d\Phi}{\omega} \cdot \eta \in C_{cl}^{1,\nu}(\Lambda_{\mathbb{R}}^{n-1,n-1})$, then
		\begin{align*}
			\restr{d^2 \mathcal{F}}{\Phi(\omega)}(\mu,\mu) = &\frac{(n!)^{\frac{n-1}{n}}}{(n-1)!} \bigg\{ \left( \int_{\mathbb{C}P^{n}} \omega^{n-1} \wedge \omega_{FS} \right) \left( \frac{1}{n-1} \left( \int_{\mathbb{C}P^{n}} \mu \wedge \omega \right)^2 - \int_{\mathbb{C}P^{n}} \mu \wedge \eta \right) \\
				& +2\left( \int_{\mathbb{C}P^{n}} \mu \wedge \omega \right) \left( \left( \int_{\mathbb{C}P^{n}} \omega^{n-1} \wedge \omega_{FS} \right) \left( \int_{\mathbb{C}P^{n}} \mu \wedge \omega \right) - \int_{\mathbb{C}P^{n}} \mu \wedge \omega_{FS} \right) \bigg\}.
		\end{align*}
	\end{teo}
 	
 	As a corollary we obtain a second variational formula over a Kähler metric. Since, the Kähler metrics are critical points for the functional $\mathcal{M}$ we do not require the trivialization to express the second variation.   
 	
 	\begin{prop}\label{SVFoverKahler}
 		Let $\omega \in \mathcal{K}^{2,\nu}_1(\mathbb{C}P^n)$, $h \in T_{g_{\omega}} \mathscr{G}^{2,\nu}(\mathbb{C}P^n)$ and $\eta=h(J\cdot, \cdot)$.  Then 
 		$$\restr{d^2 \mathcal{M}}{g_{\omega}}(h, h) = \frac{(n!)^{\frac{n-1}{n}}}{(n-2)!} \left\{\left(\int_{\mathbb{C}P^n}{\eta\wedge \omega^{n-1}}\right)^2 - \int_{\mathbb{C}P^n}{\eta \wedge \eta \wedge \omega^{n-2}}\right\}.  $$     
 	\end{prop} 
 	\begin{proof}
 		First set $\mu = \restr{d\Phi}{\omega} \cdot \eta$, then by definition \eqref{trivialization} we obtain $\mu=(n-1)\eta \wedge \omega ^{n-2}$. Hence, by Theorem \ref{svfforf} the following holds:     
 		\begin{align*}
 			\restr{d^2 \mathcal{M}}{g_\omega}(h,h) = &\frac{(n!)^{\frac{n-1}{n}}}{(n-1)!} \bigg\{ \left(\int_{\mathbb{C}P^n}{\eta\wedge \omega^{n-1}}\right)^2 - \int_{\mathbb{C}P^n}{\eta \wedge \eta \wedge \omega^{n-2}} \\
 			& +2\left( \int_{\mathbb{C}P^{n}} \mu \wedge \omega \right) \left(\int_{\mathbb{C}P^n}{\mu \wedge (\omega-\omega_{FS})} \right) \bigg\}.
 		\end{align*}
 	Where we used that $\int_{\mathbb{C}P^n} \omega^{n-1} \wedge \omega_{FS}=1$ by the normalization of the Kähler form. The result follows as in Proposition \ref{FVFforMrestrG}.  
 	\end{proof}
 
 	\section{Isosystolic Inequalities}\label{section:IsosystolicInequalities}
 	
 	Having derived the first and second variational formulas, we now begin the proof of Theorem~\ref{maintheoremA} by establishing an infinitesimal analogue of the statement. Specifically, we demonstrate that the Hessian of the mean functional over a Kähler metrics constitutes a semi-positive definite bilinear form, whose kernel coincides precisely with the directions generated by Kähler deformations.
 	
 	To this end, fix a smooth normalized Kähler metric $g_{\omega} \in \mathscr{K}(\mathbb{C}P^n)$ and define the bilinear form $\mathcal{Q}_{\omega}: T_{\omega} \mathcal{G}^{2,\nu}(\mathbb{C}P^n) \times T_{\omega} \mathcal{G}^{2,\nu}(\mathbb{C}P^n) \to \mathbb{R}$ as follows:
 	\begin{equation}\label{definitionofQ}
 		\mathcal{Q}_{\omega}(\eta,\xi) \doteq \left(\int_{\mathbb{C}P^n}\eta \wedge \omega^{n-1}\right)\left(\int_{\mathbb{C}P^n}\xi \wedge \omega^{n-1}\right) - \int_{\mathbb{C}P^n}\eta \wedge \xi \wedge \omega^{n-2}.
 	\end{equation}
 	By Proposition~\ref{SVFoverKahler}, the bilinear form $\mathcal{Q}_\omega$ coincides with the Hessian of $\restr{\mathcal{M}}{\mathscr{G}}$ (up to multiplicative constants) under the identification~\eqref{formxmetrics}.
 	 
 	The main tool we will need for study the quadratic form $\mathcal{Q}_\omega$ is the following decomposition of tangent space of $ \mathcal{G}^{2,\nu}(\mathbb{C}P^n)$ at ${\omega}$.
 	\begin{teo}\label{T_omegaG}
 		Let $g_\omega \in \mathscr{K}(\mathbb{C}P^n)$ be a smooth normalized Kähler metric, then it holds 
 		$$T_{\omega}\mathcal{G}^{2,\nu}(\mathbb{C}P^n) = T_{\omega}\mathcal{B}^{2,\nu}(\mathbb{C}P^n) \oplus \mathcal{D}_{\omega},$$
 		where $\mathcal{D}_{\omega} \doteq \left\{ dJ\beta + d^c\beta : \beta \in C^{3,\nu}(\Lambda^1_{\mathbb{R}}) \mbox{ and $\delta_g \beta = \delta_g^c \beta=0$}\right\}$.  
 	\end{teo}   
 	\begin{proof}
 		First, fix $\eta \in T_{\omega}\mathcal{G}^{2,\nu}(\mathbb{C}P^n) = \{\eta \in C^{2,\nu}(\Lambda^{1,1}_{\mathbb{R}}) : dd^c(\eta \wedge \omega^{n-2}) = 0\}$. Let $\hat{\eta} \in C_{dd^c}^{2,\nu}(\Lambda^{n-1,n-1}_{\mathbb{R}})$ be defined by 
 		$$ \hat{\eta} \doteq \eta\wedge \omega^{n-2} - \left(\int_{\mathbb{C}P^n}{\eta \wedge \omega ^{n-1}}\right)\omega ^{n-1}.$$
 		
 		The Hodge decomposition guarantee the existence of $\zeta \in C^{3,\nu}(\Lambda^{2n-3})$ and $\xi \in C^{3,\nu}(\Lambda^{2n-1})$, such that $\hat{\eta} = d \zeta + \delta_{g} \xi$. On the other hand, by Proposition 1.2.30 of \cite{huybrechts2005complex}, there exists ${\theta} \in C^{3,\nu}(\Lambda^1)$ such that $\xi=\theta \wedge \omega^{n-1}$. Now, the Kähler identities yields to
 		$$ \delta_g \xi = \left(\delta_g L_{g}^{n-2}\right) \theta = L_g^{n-1} \delta_{g} \theta - (n-1)d^c L_{g}^{n-2} \theta.$$  
 		Hence, $\hat{\eta} = f\omega^{n-1} + d \zeta - (n-1)d^c L_{g}^{n-2} \theta$, for $f=\delta_g\theta$.
 		
 		 Applying the $dd^c$ operator to $\hat{\eta}$ we obtain
 		$$0 = dd^c \hat{\eta} = dd^cf \wedge \omega^{n-1} = \left(\Delta_{g}f\right) \omega^{n},$$
 		where in the last equality we used that for every $(1,1)$-form it holds that $\alpha\wedge \omega^{n-1} = \Lambda_{g}(\alpha)\, \omega^{n}$. Therefore, $\delta_{g}\theta = 0$ and $\hat{\eta} = d \zeta - (n-1)d^c L_{g}^{n-2} \theta$. 
 		
 		On the other hand, since $\hat{\eta}$ is $J$-invariant, there exists $\lambda \in C^{4,\nu}(\Lambda^{2n-4})$ satisfying $d\zeta = -(n-1)d J\left(L_{g}^{n-2} \theta\right) + dd^c \lambda$. In conclusion, 
 		$$\eta\wedge \omega^{n-2} = \left(\int_{\mathbb{C}P^n}{\eta \wedge \omega ^{n-1}}\right)\omega ^{n-1} + dd^c \lambda - (n-1) \left(dJ\theta + d^c \theta\right)\wedge \omega^{n-2}.$$
 		
 		Again applying Proposition 1.2.30 of \cite{huybrechts2005complex}, we observe that the multiplication by $\omega^{n-2}$ is bijective when restricted to the $2$-forms. Hence, there exists $\rho$ with 
 		$$\rho \wedge \omega^{n-2} = \left(\int_{\mathbb{C}P^n}{\eta \wedge \omega ^{n-1}}\right)\omega ^{n-1} + dd^c \lambda,$$ 
 		which implies the decomposition $ \eta = \rho + dJ\beta + d^{c}\beta,$ where $\beta \doteq (1-n)\theta$. 
 		
 		Note that, by definition $\rho \in T_{\omega} \mathcal{B}^{2,\nu}(\mathbb{C}P^n)$ and also $\beta$ is $\delta_g$-closed. However, $\beta$ may not be $\delta^{c}_g$-closed.
 		
 		In order to fix this problem we observe that the map $\beta \mapsto dJ\beta + d^c \beta$ is invariant by the $C^{4,\nu}(\mathbb{C}P^n)$-action $\phi \mapsto \beta + d^c\phi$. Moreover, since $d^c$ and $\delta_g$ anti-commute this action preserve the $\delta_g$-closedness. Now, since $\int_{\mathbb{C}P^n}\delta_{g}\beta=0$, there is a unique $\phi \in C^{4,\nu}(\mathbb{C}P^n)$ satisfying the equations: 
 		\begin{equation*}
 			\begin{cases}
 				\Delta_{g} \phi = -\delta_{g}^c\beta \\
 				\int_{\mathbb{C}P^n} \phi^2 \,d \mathrm{Vol}_{g} = 1.
 			\end{cases}
 		\end{equation*}
 		In particular, $\delta_g^c(\beta + d^c\phi) = \delta^c_g\beta + \Delta_g\phi =0$. Hence, changing $\beta$ by $\beta+ d^c \phi$ if necessary, we obtain our desired decomposition.  
 		
 		It remains to prove that $T_{\omega}\mathcal{B}^{2,\nu}(\mathbb{C}P^n) \cap \mathcal{D}_{\omega} = \{0\}$. Let $\eta \in T_{\omega}\mathcal{B}^{2,\nu}(\mathbb{C}P^n) \cap \mathcal{D}_{\omega}$. By definition there exists a $\beta \in C^{3,\nu}(\Lambda^1_{\mathbb{R}})$ with  
 		\[
 		\delta_g \beta = 0, \quad \delta_g^c \beta = 0, \quad \text{and} \quad \eta = dJ\beta + d^c\beta,
 		\] 
 		such that
 		\[
 		0 = d\eta\wedge \omega^{n-2} = dd^c\beta \wedge \omega^{n-2}.
 		\]
 		
 		However, by Proposition 1.2.30 of \cite{huybrechts2005complex} we have that $\ker L_g^{n-2} = \ker \Lambda_g$, when restricted to the $3$-forms, hence by the Kähler identities 
 		\[
 		0 = \Lambda_g (dd^c \beta) = dd^c\left(\Lambda_g \beta\right) + d\delta_g \beta - \delta^c_g d^c \beta.
 		\]
 		
 		Using that $\beta$ is a $1$-form with $\delta_g \beta = 0$ and $\delta_g^c \beta = 0$, along with the fact that the Hodge Laplacian on a Kähler manifold can be written as $\Delta_g = d^c\delta^c_g + \delta^c_g d^c$, we obtain
 		\[
 		0 = \Lambda_g (dd^c \beta) = \Delta_g \beta.
 		\]
 		Since $H^1_{dR}(\mathbb{C}P^n)$ is trivial we have that $\beta=0$. Therefore, $\eta = 0$ as claimed.                                 
 	\end{proof}

	Using the techniques of \cite{junior2023balanced} we can further decompose the tangent space of the space of balanced metrics, in order to obtain the following corollary. 
	
	\begin{coro}\label{decompTGinKandB}
		 Let $g_\omega \in \mathscr{K}(\mathbb{C}P^n)$ be a smooth normalized K\"ahler metric. Then, the following decomposition holds
		 \begin{equation}\label{decompositionofTgG}
		 	T_{\omega}\mathcal{G}^{2,\nu}(\mathbb{C}P^n) = T_{\omega}\mathcal{K}^{2,\nu}(\mathbb{C}P^n) \oplus \mathcal{C}_\omega \oplus \mathcal{D}_\omega,
		 \end{equation}
		 where:
		 \begin{itemize}
		 	\item $\mathcal{C}_\omega = \{\delta_g^c \delta_g \theta \mid \theta \in C^{4,\nu}(\Lambda^4)\} \cap T_{\omega}\mathcal{B}^{2,\nu}(\mathbb{C}P^n)$;
		 	
		 	\item $\mathcal{D}_\omega \doteq \{dJ\beta + d^c\beta \mid \beta \in C^{3,\nu}(\Lambda^1_{\mathbb{R}}),\ \delta_g \beta = \delta_g^c \beta = 0\}$.
		 \end{itemize}
	\end{coro}
	\begin{proof}
		By the proof of Lemma 5.7 of \cite{junior2023balanced} we observe that $ T_{\omega}\mathcal{B}^{2,\nu}(\mathbb{C}P^n) = T_{\omega}\mathcal{K}^{2,\nu}(\mathbb{C}P^n) \oplus \mathcal{C}_\omega$, with 
		$$\mathcal{C}_\omega = \{ \delta_g \theta \mid \theta \in C^{3,\nu}(\Lambda^3)\} \cap T_{\omega}\mathcal{B}^{2,\nu}(\mathbb{C}P^n).$$
		
		But, every $\delta_g\theta \in \mathcal{C}_\omega$ is $J$-invariant, in particular it is $\delta^c_g$-closed. Hence, the conclusion follows from the $dd^c$-Lemma.
	\end{proof}
	
	For this decomposition to be effective in the analysis of the Hessian, we must first establish its $\mathcal{Q}_\omega$-orthogonality, which we prove in the following proposition.    
	
	\begin{prop}
			Let $g_\omega \in \mathscr{K}(\mathbb{C}P^n)$ be a smooth normalized Kähler metric, then: 
			\begin{enumerate}[label=\alph*),ref=(\alph*)]
				\item\label{KernelofQ} $\mathcal{Q}_\omega\left(T_{\omega}\mathcal{K}^{2,\nu}(\mathbb{C}P^n),T_{\omega}\mathcal{G}^{2,\nu}(\mathbb{C}P^n)\right)=0;$
				\item\label{Q(V,W)=0} $\mathcal{Q}_\omega(\mathcal{C}_\omega, \mathcal{D}_{\omega})=0$. 
			\end{enumerate}
	\end{prop}
	\begin{proof}
		First we prove item \ref{KernelofQ}. Consider $\eta \in T_{\omega}\mathcal{K}^{2,\nu}(\mathbb{C}P^n)$ and $\xi \in T_{\omega}\mathcal{G}^{2,\nu}(\mathbb{C}P^n)$. Applying the $dd^c$-Lemma to both Kähler and Gauduchon forms, we obtain
		 \[
		 \eta = a\omega +dd^cu, \quad \text{and} \quad \xi \wedge \omega^{n-2} = b\omega^{n-1} + dd^c \theta,
		 \]
		for a constant $a \in \mathbb{R}$, $u \in C^{4,\nu}(\mathbb{C}P^n)$ and $\theta \in C^{4,\nu}(\Lambda^{2n-3})$. Hence, by Stokes' Theorem 
		
		\begin{align*}
			\mathcal{Q}_\omega(\eta, \xi)&= \left(\int_{\mathbb{C}P^n}a\omega^n \right)\left(\int_{\mathbb{C}P^n}(\xi\wedge \omega^{n-2}) \wedge \omega \right) - \left( \int_{\mathbb{C}P^n}{(a\omega + dd^cu)\wedge \xi \wedge \omega^{n-2}} \right)\\
			&= ab \left(\int_{\mathbb{C}P^n}\omega^n \right)^2 - \left( \int_{\mathbb{C}P^n}{a \left(\xi \wedge \omega^{n-2}\right) \wedge \omega} + u\wedge dd^c\left(\xi \wedge \omega^{n-2} \right) \right)\\
			&= ab-ab =0.    
		\end{align*}

		Now we follow with the prove of item \ref{Q(V,W)=0}. Fix $\eta \in T_{\omega}\mathcal{B}^{2,\nu}(\mathbb{C}P^n)$ and $\xi = dJ\beta + d^c\beta$, with $\delta_g\beta= \delta^c_g\beta = 0$. Then
		$$\mathcal{Q}_{\omega}(\eta,\xi)=\left(\int_{\mathbb{C}P^n}\eta\wedge \omega^{n-1}\right)\left(\int_{\mathbb{C}P^n}(dJ\beta + d^c\beta)\wedge \omega^{n-1}\right) - \left( \int_{\mathbb{C}P^n}{(dJ\beta + d^c \beta)\wedge \eta \wedge \omega^{n-2}} \right). $$  
		On the other hand, $\eta$ is $J$-invariant hence $d^c(\eta \wedge \omega ^{n-2})=0$, therefore the desired result follows again by Stokes' Theorem.   
	\end{proof}

	The $\mathcal{Q}_{\omega}$-orthogonality of decomposition~\eqref{decompositionofTgG} allow us to separate the analysis of the Hessian of $\mathcal{M}$ on each component. This approach yields both the semi-positive definiteness and the kernel characterization.
	
	\begin{prop}\label{QonVandW}
			Let $g_\omega \in \mathscr{K}(\mathbb{C}P^n)$ be a smooth normalized Kähler metric. Then, 
			\begin{enumerate}[label=\alph*),ref=(\alph*)]
				\item\label{QinBal} $\mathcal{Q}_\omega(\eta, \eta) = (n-2)! \int_{\mathbb{C}P^n}{|| \Pi_{\delta,g}(\eta) ||^2_{g}\, d\mathrm{Vol}_g }$, for every $\eta \in T_{\omega}\mathcal{B}^{2,\nu}(\mathbb{C}P^n)$.
				\item\label{QinVomega} $\mathcal{Q}_{\omega}(\xi,\xi) = (n-2)!\int_{\mathbb{C}P^n} ||\xi||^2_g \, d\mathrm{Vol}_g$, for every $\xi\in \mathcal{C}_{\omega}$.    
			\end{enumerate}
		In particular, the kernel of the quadratic form $\mathcal{Q}_\omega$ is given by $T_{\omega} \mathcal{K}^{2,\nu}(\mathbb{C}P^n)$. 
	\end{prop}
	\begin{proof}
	First, the proof of item $\ref{QinBal}$ is given in Corollary 5.14 of \cite{junior2023balanced}. Now, to prove item $\ref{QinVomega}$, we use the same idea. Let
	\begin{align*}
		\mathcal{RH}: C^{2,\nu}{(\Lambda^{2})} \times C^{2,\nu}{(\Lambda^{2})} &\to C^{2,\nu}{(\Lambda^{2n})}\\
		(\eta, \xi) &\mapsto - \eta \wedge \xi \wedge \omega^{n-2},
	\end{align*}
	denote the \emph{Riemann-Hodge pairing} associated with $\omega$ (see also Definition B.5 of \cite{junior2023balanced}). Then, for every $\xi = dJ\beta + d^c \beta \in \mathcal{C}_\omega$, we have
	\begin{align*}
		\mathcal{Q}_{\omega}(\xi,\xi)&=\left(\int_{\mathbb{C}P^n} \xi\wedge \omega^{n-1}\right)^2-\left(\int_{\mathbb{C}P^n}{\xi\wedge \xi \wedge \omega^{n-2}} \right) \\
		&=\int_{\mathbb{C}P^n}{\mathcal{RH}(\xi, \xi)},
	\end{align*}
	Here, we used Stokes' Theorem to infer that $\int_{\mathbb{C}P^n} \xi\wedge \omega^{n-1}=0$. 
	
	On the other hand, by the Kähler identities and the fact that $\Lambda_g(\beta) = 0$, we obtain
	\begin{align*}
		\Lambda_g(\xi) &= (\Lambda_g \circ d)J\beta + (\Lambda_g \circ d^c)\beta \\
		&=-\delta^c_gJ\beta + \delta_g\beta\\
		&=0,
	\end{align*}
	where in the last equality we used that $\delta_g \beta = 0$. Hence, $\xi \in C^{2,\nu}(\Lambda_{\mathbb{R}}^{1,1})$ is a primitive form, which allows us to apply the \emph{Riemann-Hodge bilinear relations} (see Theorem B.6 in \cite{junior2023balanced}). Therefore,
	$$ \mathcal{Q}_{\omega}(\xi, \xi ) = \int_{\mathbb{C}P^n}{\mathcal{RH}(\xi, \xi)}= (n-2)! \int_{\mathbb{C}P^n} || \xi ||^2_g \, d\mathrm{Vol}_g.$$
 	\end{proof}
	
	\subsection*{Isosystolic Inequality for Gauduchon metrics}
	Proposition~\ref{QonVandW} establishes an infinitesimal analog of Theorem~\ref{maintheoremA}. In this section, we use Taylor expansions to lift this infinitesimal result to a local statement. The central technical task is to express the mean functional~\eqref{meanfunctional} in appropriate coordinates in order to properly apply the Taylor expansion.
        
	Prior to construct the aforementioned coordinates, we recall that under the global chart $\Phi: \mathcal{G}^{2,\nu}(\mathbb{C}P^n) \to C^{2,\nu}_{dd^c}(\Lambda^{n-1,n-1}_+)$, defined in~\eqref{trivialization}, the mean functional $\mathcal{M}$ on $\mathscr{G}^{2,\nu}(\mathbb{C}P^n)$ is expressed as $\mathcal{F} = \mathcal{M} \circ \Phi$.
	
	\begin{lema}\label{summarizinglemmablc} 
		Let $g_\omega \in \mathscr{K}(\mathbb{C}P^n)$ be a smooth normalized Kähler metric. There exist open neighborhoods $U \subset \mathcal{K}^{2,\nu}(\mathbb{C}P^n)$ of $\omega$ and $V\subset \mathcal{C}_\omega \oplus \mathcal{D}_{\omega}$ of $0$, along with a smooth diffeomorphism $\rho:U\times V \to \rho(V\times U)\subset  C_{dd^c}^{2,\nu}(\Lambda_+^{n-1,n-1})$, such that the map $F \doteq \mathcal{F}\circ \rho:U \times V  \to \mathbb{R}$ satisfies the following properties:  
		\begin{enumerate}[label=\alph*),ref=(\alph*)]
			\item\label{summarizinglemmablc1}  ${F}$ is constant over the set $U\times \{0\}$.
			\item\label{summarizinglemmablc2}  $\restr{d {F}}{\omega}\equiv 0$, for every $\omega \in U$.
			\item\label{summarizinglemmablc3}  The Hessian map of $F$ at $(\omega,0)$ is a symmetric, semi-positive definite bilinear form. Moreover, its kernel is given by $T_{\omega}\mathcal{K}^{2,\nu}(\mathbb{C}P^n)$.
			\item\label{summarizinglemmablc4}  Given $(\theta,\lambda) \in U \times V$, the restriction $\restr{d^2 {F}}{(\theta,\lambda)}: \mathcal{C}_\omega \oplus \mathcal{D}_{\omega} \times \mathcal{C}_\omega \oplus \mathcal{D}_{\omega} \to \mathbb{R}$ is given by 
			\begin{equation}\label{transfooftheHess}
				\restr{d^2 {F}}{(\theta, \lambda)}\left(\eta,\eta\right)= \restr{d^2\mathcal{F}}{\rho(\theta,\lambda)}\left( \restr{d\Phi}{\omega}\cdot \eta,\restr{d\Phi}{\omega}\cdot \eta \right).
			\end{equation}
			In particular, $\restr{d^2F}{\omega}\left( (\zeta,\xi),(\zeta,\xi)\right)=(n!)^{\frac{n-1}{n}} \left\{ ||\zeta||^2_{L^2(g)} + ||\xi||^2_{L^2(g)}  \right\}$, for every $\zeta \in \mathcal{C}_\omega$ and $\xi \in \mathcal{D}_\omega$.
		\end{enumerate} 
	\end{lema}
	\begin{proof}
	Fix $g_\omega \in \mathscr{K}(\mathbb{C}P^n)$ and define $\mathcal{E}_{\omega} = \mathcal{C}_\omega \oplus \mathcal{D}_{\omega}$. Consider the smooth map
	\begin{align}\label{defofrho}
		\begin{split}
			\rho \colon \mathcal{K}^{2,\nu}(\mathbb{C}P^n) \times \mathcal{E}_{\omega} &\to C^{2,\nu}_{dd^c}(\Lambda^{n-1,n-1}_{\mathbb{R}}), \\
		(\theta, \lambda) &\mapsto \Phi(\theta) + d\Phi|_{\omega}\cdot \lambda,
		\end{split} 
	\end{align}
	 By Corollary~\ref{decompTGinKandB}, we can identify the differential $d\rho|_{(\omega,0)}$ with $d\Phi|_{\omega}$. Since $\Phi$ is a diffeomorphism, the map
	\[
	d\rho|_{(\omega,0)} \colon T_{\omega} \mathcal{G}^{2,\nu}(\mathbb{C}P^n) \to C^{2,\nu}_{dd^c}(\Lambda^{n-1,n-1}_{\mathbb{R}})
	\]
	is a linear isomorphism of Banach spaces. Applying the inverse function theorem for Banach spaces, we conclude that $\rho$ is a local diffeomorphism at $(\omega,0)$. Consequently, there exist neighborhoods $U \subset \mathcal{K}^{2,\nu}(\mathbb{C}P^n)$ of $\omega$ and $V \subset \mathcal{E}_{\omega}$ of $0$ such that the restriction
	\[
	\rho \colon U \times V \to \rho(U \times V) \subset C^{2,\nu}_{dd^c}(\Lambda^{n-1,n-1}_+)
	\]
	is a diffeomorphism onto its image.

	We now establish the claimed properties of the diffeomorphism $\rho$. 
	Item~\ref{summarizinglemmablc1} follows from Proposition~\ref{systolexM,MsigmaonGaud}, 
	and Item~\ref{summarizinglemmablc2} follows directly from Proposition~\ref{FVFforMrestrG}. Moreover, Item~\ref{summarizinglemmablc3} follows from the fact that $\Phi(\omega)$ is a critical point for $\mathcal{F}$ together with Proposition~\ref{QonVandW}.    
	
	Finally, to establish Item~\ref{summarizinglemmablc4}, take $(\theta, \lambda) \in U \times V$ and $\eta \in \mathcal{E}_\omega$. The definition of $\rho$ in~\eqref{defofrho} implies that for all sufficiently small $t \in \mathbb{R}$,
	\[
	F(\theta, \lambda + t\eta) = \mathcal{F}\left( \rho(\theta, \lambda) + t \, d\Phi|_{\omega} \cdot \eta \right).
	\]
	Differentiating this equality twice with respect to $t$ at $t = 0$ yields the identity~\eqref{transfooftheHess}. In particular, Proposition~\ref{QonVandW} implies the stated expression for the Hessian at $\omega$.       
	\end{proof}

	The preceding Lemma, specifically Item~\ref{summarizinglemmablc4}, establishes that the Hessian of the normalized holomorphic $(n-1)$-systole lacks coercivity in the $C^{2,\nu}$-topology. Consequently, a direct application of the Taylor expansion argument is not possible. To proceed with the argument we must first demonstrate continuity of the mapping
	\begin{align*}
		\mathcal{K}^{2,\nu}(\mathbb{C}P^n) &\to \mathcal{S}_2\left(C^{2,\nu}(\Lambda^{n-1,n-1})\right) \\
		\omega &\mapsto \left. d^2\mathcal{F} \right|_{\Phi(\omega)},
	\end{align*}
	where the space of symmetric bilinear forms $\mathcal{S}_2\left(C^{2,\nu}(\Lambda^{n-1,n-1})\right)$ is equipped with the $L^2$-norm topology. 
	
	In view of the expression derived for the map $\omega \mapsto \restr{d^2\mathcal{F}}{\Phi(\omega)}$ in Theorem~\ref{svfforf}, we can carry the proof of Lemma 5.16 in~\cite{junior2023balanced} without any significant changes in order to prove the following. 
	\begin{lema}\label{L2boundofHess}
		Let $g_\omega \in \mathscr{K}(\mathbb{C}P^n)$ be a smooth normalized Kähler metric. Then, there exist a neighborhood $\mathcal{N} \subset \mathcal{G}^{2,\nu}(\mathbb{C}P^n)$ of $\omega$, in the $C^{2,\nu}$-topology, such that for each $\theta \in \mathcal{N}$ and $\eta \in T_{\omega}\mathcal{G}^{2,\nu}(\mathbb{C}P^n)$ the inequality
		\begin{equation*}
			\left| \restr{d^2 \mathcal{F} }{\Phi(\omega)}\left(\mu,\mu\right) - \restr{d^2 \mathcal{F} }{\Phi(\theta)}\left(\mu,\mu\right) \right| \leq \frac{(n!)^{\frac{n-1}{n}}}{2} \,||\eta||^2_{L^2(g)},
		\end{equation*}
		holds, where $\mu= \restr{d \Phi}{\omega} \cdot \eta \in C^{2,\nu}_{dd^c}\big(\Lambda^{n-1,n-1}_{\mathbb{R}} \big)$. 
	\end{lema}

	Now we move on to prove Theorem~\ref{maintheoremA}, which we rewrite below in terms of the functional $\mathcal{F}:  C^{2,\nu}_{dd^c}(\Lambda^{n-1,n-1}_+) \to \mathbb{R}$.    
	
	\begin{teo}\label{maintheorem1inF}
		There is an open set $ {\Phi}\left(\mathcal{K}^{2,\nu} (\mathbb{C}P^n)\cap \Omega^{1,1}(\mathbb{C}P^n)\right) \subset \mathcal{U} \subset C_{dd^c}^{2,\nu}\big(\Lambda_{+}^{n-1,n-1}\big)$, in the ${C}^{2,\nu}$-topology, such that for every form $\lambda \in \mathcal{U}$ 
		$$\mathcal{F}(\lambda) \geq \mathcal{F}(\omega_{FS}^{n-1}).$$
		Moreover, $\lambda \in \mathcal{U}$ satisfies the equality if and only if $\lambda \in {\Phi}\left(\mathcal{K}^{2,\nu}(\mathbb{C}P^n)\right)\cap \mathcal{U}$.
	\end{teo}
	\begin{proof}
		First we fix $g_{\omega} \in \mathscr{K}(\mathbb{C}P^n)$ a normalized smooth Kähler metric. Now let 
		$$ \rho:U\times V  \subset \mathcal{K}^{2,\nu}(\mathbb{C}P^n)\times \mathcal{C}_\omega \oplus \mathcal{D}_\omega \to \rho(V\times U)\subset  C_{dd^c}^{2,\nu}(\Lambda_+^{n-1,n-1})$$ 
		be the coordinates constructed in Lemma~\ref{summarizinglemmablc}, and also set $F = \mathcal{F} \circ \rho : U \times V \to \mathbb{R}$. 
		
		Without loss of generality we can suppose that $U$ and $V$ are open convex sets such that $\mathcal{W}_\omega \doteq \rho(U\times V)$ is contained in $\Phi(\mathcal{N})$, for $\mathcal{N}$ neighborhood of $\omega$ satisfying the properties of Lemma~\ref{L2boundofHess}.   
		
		The second order Taylor expansion with Lagrange remainder around $\theta \in U$ implies that for every $\eta \in V$ there is a constant $c=c(\eta) \in (0,1)$ such that:
		
		\begin{align*}
			{F}(\theta,\eta) &={F}(\theta) + \restr{d{F}}{\theta}\cdot \eta + \frac{1}{2}\restr{d^2 {F}}{(\theta,c \eta)}(\eta,\eta)  \\
			& = {F}(\theta) + \restr{d{F}}{\theta}\cdot \eta + \frac{1}{2}\restr{d^2{F}}{\omega}(\eta,\eta) + \frac{1}{2}\left( \restr{d^2 {F}}{(\theta,c \eta)}(\eta,\eta) - \restr{d^2 {F}}{\omega}(\eta,\eta) \right)\\
			&= \mathcal{F}(\omega^{n-1}_{FS}) + \frac{1}{2}\restr{d^2{F}}{\omega}(\eta,\eta) + \frac{1}{2}\left( \restr{d^2 {F}}{(\theta,c \eta)}(\eta,\eta) - \restr{d^2 {F}}{\omega}(\eta,\eta) \right).   
		\end{align*}    
	Were in the last equality we have used Items~\ref{summarizinglemmablc1} and \ref{summarizinglemmablc2} of Lemma~\ref{summarizinglemmablc}. 
	
	Since $\mathcal{C}_{\omega} \subset \operatorname{im}(\delta_g \delta^c_g)$ and $\mathcal{D}_{\omega} \subset \operatorname{im}(dJ + d^c)$, we conclude that $\mathcal{C}_{\omega}$ and $\mathcal{D}_{\omega}$ are $L^2(g)$-orthogonal. In particular, by Item~\ref{summarizinglemmablc4} of Lemma~\ref{summarizinglemmablc}, we obtain that
	$$\restr{d^2{F}}{\omega}(\eta,\eta) = (n!)^{\frac{n-1}{n}}||\eta||^2_{L^2(g)}. $$  
	
	In particular, combining equation~\eqref{transfooftheHess} with Lemma~\ref{summarizinglemmablc}, we obtain that for every element $\lambda = \rho(\theta, \eta) \in \mathcal{W}_\omega$, 
	$$\mathcal{F}(\lambda) \geq \mathcal{F}(\omega^{n-1}_{FS}) + \frac{(n!)^{\frac{n-1}{n}}}{4}||\eta||^2_{L^2(g)}.$$ 
	And hence the desired inequality in the open set $\mathcal{W}_\omega$. Moreover, equality holds if and only if $\lambda=\rho(\theta, 0) \in \Phi(\mathcal{K}^{2,\nu}(\mathbb{C}P^n))$. 
	
	The open set $\mathcal{U}$ is obtained by taking the union of $\mathcal{W}_\omega$ along the set of Kähler forms, noting that we can drop the normalization hypotheses via the homothety invariance of the normalized holomorphic $(n-1)$-systole.  
	\end{proof}
	   
	\section{Applications}\label{section:Applications}
	
	This section is devoted to applications of the results derived in Sections~\ref{section:VariationalFormulae} and~\ref{section:IsosystolicInequalities}. In particular, we establish the proofs of Corollaries~\ref{maintheoremB}, \ref{maintheoremC} and Theorem~\ref{maintheoremE}.
	
    In our first applications, we prove that every tangent direction to the space of Hermitian metrics can be integrated into a $1$-parameter perturbation of  the Fubini-Study metric, along which the normalized holomorphic $(n-1)$-systole remains constant, is strictly increasing, or is strictly decreasing.
    
    \begin{prop}
    	For any $h \in \Omega^{1,1}(\mathbb{C}P^n)$, there exists a $1$-parameter family of Hermitian metrics 
    	$g_t \in \mathscr{H}(\mathbb{C}P^n)$ with the initial conditions $g_0 = g_{\text{FS}}$ and $\left.\frac{d}{dt}g_t\right|_{t=0} = h$, 
    	satisfying exactly one of the following conditions:
    	\begin{enumerate}[label=\alph*),ref=(\alph*)]
    		\item $\rho(g_t) = \rho(g_{\text{FS}})$ for all $t$;
    		\item $\rho(g_t) > \rho(g_{\text{FS}})$ for all $t \neq 0$;
    		\item $\rho(g_t) < \rho(g_{\text{FS}})$ for all $t \neq 0$.
    	\end{enumerate}   
    \end{prop}
	\begin{proof}
	First, suppose $h \notin T_{g_{FS}} \mathscr{G}(\mathbb{C}P^n)$. By Corollary~\ref{firstorderdesc}, the family of Hermitian metrics $g_t \doteq g_{FS} + t h$ satisfies $\rho(g_t) < \rho(g_{FS})$ for all sufficiently small $t \neq 0$.
	
	Now set $\eta = h(\cdot, J\cdot)$ and assume $\eta \in T_{\omega_{FS}} \mathcal{G}^{2,\nu}(\mathbb{C}P^n)$. If $\eta \in T_{\omega_{FS}} \mathcal{K}^{2,\nu}(\mathbb{C}P^n)$, then the smooth family of metrics $g_t$ defined by their fundamental forms as $\omega_{g_t} = \omega_{FS} + t\eta$ consists of Kähler metrics. In particular, the normalized holomorphic $(n-1)$-systole remains constant for all $t$.
	
	Finally, suppose $\eta \in T_{\omega_{FS}} \mathcal{G}^{2,\nu}(\mathbb{C}P^n) \setminus T_{\omega_{FS}} \mathcal{K}^{2,\nu}(\mathbb{C}P^n)$. By Corollary~\ref{decompTGinKandB}, we have the decomposition
	\[
	T_{\omega_{FS}} \mathcal{G}^{2,\nu}(\mathbb{C}P^n) = T_{\omega_{FS}} \mathcal{K}^{2,\nu}(\mathbb{C}P^n) \oplus \mathcal{C}_{\omega_{FS}} \oplus \mathcal{D}_{\omega_{FS}}
	\]
	and it is an straightforward consequence of the proof of that corollary that the decomposition preserves the smooth regularity of the forms. Thus, there exist smooth forms $\zeta \in T_{\omega_{FS}} \mathcal{K}^{2,\nu}(\mathbb{C}P^n)$ and $\xi \in \mathcal{C}_{\omega_{FS}} \oplus \mathcal{D}_{\omega_{FS}}$ such that $\eta = \zeta + \xi$.    
	
	Define $\omega_t = \omega_{FS} + t\zeta + t\xi$. For sufficiently small $t$, $\omega_t$ lies in the coordinate domain from Lemma~\ref{summarizinglemmablc}. By the proof of Theorem~\ref{maintheorem1inF}, for $g_t = g_{\omega_t}$ we have
	$$\rho(g_t) \geq \rho(g_{FS})+tc(n)\,||\xi||^2_{L^2(g_{FS})}>\rho(g_{FS})$$
	for all $t\neq 0$, where $c(n) > 0$ is a dimensional constant. This completes the proof.  
	\end{proof}

	As our next application, we derive Theorem~\ref{maintheoremE} by combining P. Gauduchon's seminal work on the existence of Gauduchon metrics (stated below) with M. Gromov's observation that the integral geometric formula~\eqref{IGF} represents an open condition.    
	
	\begin{teo}[\cite{Gauduchon77}]\label{GauduchonThm}
		In each conformal class of Hermitian metrics on a closed complex manifold of complex dimension $n \geq 2$ there exists a Gauduchon
		metric. Moreover, this Gauduchon metric is unique, up to homothety.
	\end{teo} 
 	
	For our purposes, it will be useful to give a sketch of the proof in $\mathbb{C}P^n$. Fix $g \in \mathscr{H}(\mathbb{C}P^n)$, let $\omega$ be its fundamental form, and consider the following differential operator 
	$$L_g(f) \doteq \langle dd^c(f\omega^{n-1}),\omega^{n}\rangle_g.$$
	The unique (normalized) Gauduchon metric conformal to $g$ is given by $\phi^{\frac{1}{n-1}} g$, where $\phi$ is the unique solution to:
	\begin{equation}\label{harmoniceq}
		\begin{cases}
			L_g (\phi) = 0 \\
			\dashint_{\mathbb{C}P^n} \phi^2 \,d \mathrm{Vol}_{g_{FS}} = 1.
		\end{cases}
	\end{equation}   
	The differential operator $L_g$ is a second order elliptic differential operator, hence existence and regularity for solutions of~\eqref{harmoniceq} are guarantee by Schauder estimates and classical theory. Moreover, uniqueness and positive of the solution are given studying its adjoint together with the maximum principle.     
	 
	 To summarize this construction, we define the projection map 
	 \begin{align*}
	 	\Pi_{\mathscr{G}}: \mathscr{H}(\mathbb{C}P^n) &\to \mathscr{G}(\mathbb{C}P^n) \\
	 	g &\mapsto \phi^{\frac{1}{n-1}} g,
	 \end{align*}
	 where $\phi$ denotes the unique solution to equation~\eqref{harmoniceq} for the metric $g$. Applying Schauder estimates, we establish that $\Pi_{\mathscr{G}}$ is continuous in the $C^{\infty}$-topology.
	 
	 The next proposition establishes that the integral geometric formula~\eqref{IGF} constitutes an open condition, which is based on the arguments developed in \cite[\S4.A.2]{gromov_systole}. 
	 
	 \begin{prop}\label{geneGIF}
	 	In a $C^{\infty}$-neighborhood of the Fubini-Study metric in $\mathscr{G}(\mathbb{C}P^n)$, every metric $g$ admits a positive Radon measure $d\mu_g$ for which the integral geometric formula
	 	\begin{equation}\label{IGFforg}
	 		\int_{\mathbb{C}P^n}\left(\dashint_{{\Sigma}_{\sigma}}f \,d \mathrm{A}_g \right) d\mu_g = \int_{\mathbb{C}P^n}f \, d \mathrm{Vol}_g
	 	\end{equation}
	 	holds identically for $f \in C^{\infty}(\mathbb{C}P^n)$.   
	 \end{prop}
 	\begin{proof}
 		Let $g$ be a Gauduchon metric. Generalizing the ideas from Section~\ref{section:ClassificationofGauduchonMetrics}, we define the \emph{Radon transform associated to $g$}:
 		\begin{align*}
 			\mathcal{R}_g : C^{\infty}(\mathbb{C}P^n) &\to C^{\infty}(\mathbb{C}P^n) \\
 			f &\mapsto \mathcal{R}_g(f)(\sigma) = \dashint_{\Sigma_\sigma} f  d\mathrm{A}_{g}.
 		\end{align*}
 		 We similarly define the dual Radon transform $\mathcal{R}^*_g : C^{\infty}(\mathbb{C}P^n) \to C^{\infty}(\mathbb{C}P^n)$. Which is well defined by means of the coarea formula applied to the double fibration~\eqref{doublefibration}. 
 		
 		We claim that if $g$ admits a positive function $\psi_g \in C^{\infty}(\mathbb{C}P^n)$ satisfying
 		\begin{equation}\label{dualRadonissurj}
 			\mathcal{R}^{*}_g(\psi_g) = 1,
 		\end{equation}
 		then the positive Radon measure $d\mu_g = \psi_g  d\mathrm{Vol}_g$ satisfies the integral geometric formula~\eqref{IGFforg}. Indeed, for any $f \in C^{\infty}(\mathbb{C}P^n)$:
 		\begin{align*}
 			\int_{\mathbb{C}P^n} \left( \dashint_{\Sigma_\sigma} f  d\mathrm{A}_g \right) d\mu_g 
 			&= \int_{\mathbb{C}P^n} \mathcal{R}_g(f)  \psi_g  d\mathrm{Vol}_g \\
 			&= \int_{\mathbb{C}P^n} f  \mathcal{R}^*_g(\psi_g)  d\mathrm{Vol}_g \\
 			&= \int_{\mathbb{C}P^n} f  d\mathrm{Vol}_g,
 		\end{align*}
 		Therefore, it remains to establish that in some $C^{\infty}$-neighborhood of the Fubini-Study metric, every metric admits a positive function $\psi_g$ satisfying~\eqref{dualRadonissurj}.  
 		
 	To establish its existence we prove surjectivity of the operator $\mathcal{T}_g \doteq \mathcal{R}_g^* \circ \mathcal{R}_g$. The double fibration~\eqref{doublefibration} induces bijections via the projections
 	\begin{equation*}\label{diffdoublefibration}
 		\begin{xy}\xymatrix{
 				& N\mathcal{I}-0 \ar[dl]_{\kappa} \ar[dr]^{\upsilon} & \\
 				T\mathbb{C}P^{n} & & T\mathbb{C}P^{n}
 		}\end{xy}
 	\end{equation*}
 	This allows us to apply the general theory of Radon transforms from \cite[Chapter~VI, \S6]{GuilleminGeoAsym}, which guarantees that $\mathcal{T}_g$ is an elliptic pseudodifferential operator of order $2-2n$. 
 	
 	Since $\mathbb{C}P^n$ is simply connected, for any integer $s > 0$ the operator $\mathcal{T}_g$ extends continuously to a Fredholm operator 
 	\[
 	\mathcal{T}_g: W^{s,2}(\mathbb{C}P^n,g_{FS}) \to W^{s+2n-2,2}(\mathbb{C}P^n,g_{FS})
 	\]
 	with index zero (see \cite[\S 3.2.4--3.2.5]{Egorov_Schulze}). Fix $s$ sufficiently large so that $W^{s,2}(\mathbb{C}P^n,g_{FS})$ embeds continuously into $C^{0}(\mathbb{C}P^n)$.

 	The Helgason's inversion formula \cite[Chapter 3, Theorem 2.2]{HelgasonBook1999} combined with elliptic regularity implies that $\mathcal{T}_{FS}$ is injective. Since it has index zero, $\mathcal{T}_{FS}$ is an isomorphism. Moreover, Eric T. Quinto's computation of the symbol of $\mathcal{T}_g$ in \cite{Quinto81} shows that the map
 	\[
 	g \mapsto \left\| \mathcal{T}_g - \mathcal{T}_{FS} \right\|_{\mathscr{L}(W^{s,2}, W^{s+2n-2,2})}
 	\]
 	is continuous on $\mathscr{G}(\mathbb{C}P^n)$ in the $C^{\infty}$-topology. Consequently, there exists a neighborhood $\mathcal{U}$ of $g_{FS}$ where $\mathcal{T}_g$ is surjective. 
 	
 	Therefore, for each $g \in \mathcal{U}$, there exists $\varphi_g \in W^{s,2}(\mathbb{C}P^n,g_{FS})$ satisfying 
 	\begin{equation}\label{Tphi=1}
 		\mathcal{T}_g(\varphi_g) = 1.
 	\end{equation}
 	Elliptic regularity applied to equation~\eqref{Tphi=1} implies that $\varphi_g$ is smooth, and by our choice of $s>0$ shrinking $\mathcal{U}$ if necessary, we ensure $\varphi_g > 0$. Defining $\psi_g \doteq \mathcal{R}_{g}(\varphi_g)$, we obtain a smooth positive function satisfying $\mathcal{R}^*_g(\psi_g) = 1$, which completes the proof.            
 	\end{proof}
 	
 	\begin{proof}[Proof of Theorem~\ref{maintheoremE}]
 		Let $\mathcal{U}$ be the neighborhood of the Fubini-Study metric obtained by intersecting the neighborhoods from Theorem~\ref{maintheoremA} and Proposition~\ref{geneGIF}. The proof of Item~\ref{mainthrm2itemb} follows immediately from Theorem~\ref{maintheoremA}, it remains to establish Item~\ref{mainthrm2itema}.
 		
 		Consider $g \in \Pi_{\mathscr{G}}^{-1}(\mathcal{U})$. By definition, there exists a smooth positive function $f \in C^{\infty}(\mathbb{C}P^n)$ such that $g = f\, \hat{g}$, for $\hat{g} \doteq \Pi_{\mathscr{G}}(g)$. Since $\hat{g} \in \mathcal{U}$, it admits a Radon measure satisfying the integral geometric formula~\eqref{IGFforg}. 
 		
 		Following the argument in \cite[Theorem A.2]{junior2023balanced}, we obtain:
 		\[
 		\frac{\inf_{\sigma \in \mathbb{C}P^n} \mathrm{Area}(\Sigma_\sigma, g)}{\mathrm{Vol}(g)^{\frac{n-1}{n}}} 
 		\leq 
 		\frac{\mathrm{Sys}^{\mathrm{Hol}}_{n-1}(\hat{g})}{\mathrm{Vol}(\hat{g})^{\frac{n-1}{n}}}.
 		\]
 		As each $\Sigma_\sigma$ is a homologically non-trivial holomorphic submanifold, we conclude $\rho(g) \leq \rho(\hat{g})$, which completes the proof of Item~\ref{mainthrm2itema}.  
 	\end{proof} 
 
 	 Now, we move to the proof of Corollary~\ref{maintheoremB}, which is a consequence of Theorem~\ref{maintheoremA}, Corollary~\ref{decompTGinKandB}, and Gromov-Berger's result on the local maximality of the Fubini-Study metric for the normalized 2-systole functional on $\mathbb{C}P^2$ (see \cite[Section 0.2.B]{gromov_pshol}). 
 	\begin{proof}[Proof of Corollary~\ref{maintheoremB}]
 		By Gromov-Berger Theorem there exist a neighborhood $\mathcal{U}' \subset \mathscr{G}(\mathbb{C}P^2)$ of the  Fubini-Study metric, such that, within this set $g_{FS}$ is a maximum for the normalized $2$-systole. Now, let $\mathcal{U}'' \subset  \mathscr{G}(\mathbb{C}P^2)$ be the neighborhood of $g_{FS}$ given by Theorem~\ref{maintheoremA}. We claim that the set $\mathcal{U} = \mathcal{U}'  \cap \mathcal{U}''$ satisfies the desired property. 
 		
 		In fact, let $g \in \mathcal{U}$ and assume that the $2$-systole is achieved by a holomorphic $1$-chain. Therefore, by the construction of the set $\mathcal{U}$ we have the following string of inequalities:    
 		$$\frac{\mathrm{Sys}_{2}(g)}{\mathrm{Vol}(g)^{\frac{1}{2}}}\leq \frac{\mathrm{Sys}_{2}(g_{FS})}{\mathrm{Vol}(g_{FS})^{\frac{1}{2}}} \leq \frac{\mathrm{Sys}^{\mathrm{Hol}}_{1}(g)}{\mathrm{Vol}(g)^{\frac{1}{2}}}\leq \frac{\mathrm{Sys}_{2}(g)}{\mathrm{Vol}(g)^{\frac{1}{2}}}. $$
 		Consequently, $\rho(g)=\rho(g_{FS})$. Then by the rigidity of Theorem~\ref{maintheoremA} the Gauduchon metric $g$ is Kähler as desired.   
 		
 		Moreover, fix $\beta \in \Omega^{1}(\mathbb{C}P^2)$ with $\delta\beta = 0$. Following the argument in Theorem~\ref{T_omegaG}, we may assume $\delta^c\beta = 0$. Define the $(1,1)$-form $\eta = dJ\beta + d^c\beta$ and the symmetric tensor $h \doteq \eta(J\cdot, \cdot)$. Hence, by Corollary~\ref{decompTGinKandB} the family of Gauduchon metrics $g_t \doteq g_{FS} + th$ is non-Kähler. Consequently, its $2$-systole cannot by realized by a holomorphic chain.               
 	\end{proof}
 
 	We conclude this section by presenting the proof of Corollary~\ref{maintheoremC}, that follows from a combination of the previous proof and the existence of minimal branched immersions $2$-spheres provided by the seminal work of J. Sacks and K. Uhlenbeck~\cite{Sacks_Uhlenbeck81}.   
 	\begin{proof}[Proof of Corollary~\ref{maintheoremC}]
 		In the subsection \textit{Small perturbations of $\mathbb{C}P^2$} of~\cite[\S4.A.2]{gromov_systole}, M. Gromov establishes the existence of a $C^{\infty}$-neighborhood $\mathcal{U}' \subset \mathrm{Riem}(\mathbb{C}P^2)$ of the Fubini-Study metric such that for every $g \in \mathcal{U}'$, there exists a smooth embedding $j: \mathbb{C}P^1 \to \mathbb{C}P^2$ homologous to a linear equator, satisfying:  
 		\begin{equation}\label{isosysforj}
 			\frac{\mathrm{Area}(\mathbb{C}P^1,j^*g)}{\mathrm{Vol}(g)^{1/2}} \leq \frac{\mathrm{Sys}_2(g_{FS})}{\mathrm{Vol}(g_{FS})^{1/2}}.
 		\end{equation}  
 		
 		Meanwhile, Theorem~\ref{maintheoremA} yields a neighborhood $\mathcal{U}'' \subset \mathscr{G}(\mathbb{C}P^2)$ where $\rho(g) > \rho(g_{FS})$ for all non-Kähler metric $g \in \mathcal{U}''$, moreover we can assume that every metric in the neighborhood $\mathcal{U}''$ has positive sectional curvature. We claim that every non-Kähler metric $g \in \mathcal{U} \doteq \mathcal{U}' \cap \mathcal{U}''$ admits a stable minimal conformal branched immersion $f_g: \mathbb{C}P^1 \to (\mathbb{C}P^2,g)$ that is neither holomorphic nor anti-holomorphic.  
 		
 		In fact, fix such $g$ and let $j$ be as before. Once that $\mathbb{C}P^2$ is simple connected the Hurewicz's Theorem implies that $\pi_2(\mathbb{C}P^2)$ and $H_2(\mathbb{C}P^2,\mathbb{Z})$ are isomorphic. Moreover, since any linear equator generates $H_2(\mathbb{C}P^2,\mathbb{Z})$ we can conclude that $j$ generated $\pi_2(\mathbb{C}P^2)$. Therefore, by~\cite[Theorem 5.9]{Sacks_Uhlenbeck81}, there exists a stable minimal conformal branched immersion $f_g: \mathbb{C}P^1 \to \mathbb{C}P^2$, such that: 
 		$$\mathrm{Area}(\mathbb{C}P^1,f_g^*g) \leq \mathrm{Area}(\mathbb{C}P^1,j^*g).$$		
 		Inequality~\eqref{isosysforj} then gives:  
 		\begin{equation}\label{mainBeq1}
 			\frac{\mathrm{Area}(\mathbb{C}P^1,f_g^*g)}{\mathrm{Vol}(g)^{1/2}} \leq \frac{\mathrm{Sys}_2(g_{FS})}{\mathrm{Vol}(g_{FS})^{1/2}}.
 		\end{equation}
 	 
 		It remain to show that $f_g$ is neither holomorphic or anti-holomorphic. Suppose by contradiction that was the case, the \emph{Proper Mapping Theorem}~\cite[p.395]{Griffiths_Harris94} would imply that $f_g(\mathbb{C}P^1)$ is holomorphic $1$-chain, and subsequently:  
 		\begin{equation}\label{mainBeq2}
 			\mathrm{Sys}_1^{\mathrm{Hol}}(g)\leq \mathrm{Area}(\mathbb{C}P^1,f_g^*g).
 		\end{equation}
 
 		Finally, inequalities~\eqref{mainBeq1} and~\eqref{mainBeq2} together with Theorem~\ref{maintheoremA} would then imply:  
 		\[
 		\frac{\mathrm{Area}(\mathbb{C}P^1,f_g^*g)}{\mathrm{Vol}(g)^{1/2}} \leq \frac{\mathrm{Sys}_2(g_{FS})}{\mathrm{Vol}(g_{FS})^{1/2}} < \frac{\mathrm{Sys}_1^{\mathrm{Hol}}(g)}{\mathrm{Vol}(g)^{1/2}} \leq \frac{\mathrm{Area}(\mathbb{C}P^1,f_g^*g)}{\mathrm{Vol}(g)^{1/2}},
 		\] 		          
 		leading to our desired contradiction.
  	\end{proof}
 
 	\bibliographystyle{halpha}
	\bibliography{refs}
\end{document}